\documentclass[final,3p]{elsarticle}
\usepackage[latin1]{inputenc}
 \usepackage{graphics}
 \usepackage{graphicx}
 \usepackage{epsfig}
\usepackage{amssymb}
 \usepackage{amsthm}
 \usepackage{lineno}
 \usepackage{amsmath}
\usepackage{color}
   \numberwithin{equation}{section}
\usepackage{mathrsfs}

\journal{ } 

\newtheorem{thm}{Theorem}[section]

\newtheorem{lem}[thm]{Lemma}

\newtheorem{defn}[thm]{Definition}
\newtheorem{rem}[thm]{Remark}

\biboptions{sort&compress,square}
\allowdisplaybreaks
\begin{document}
\begin{frontmatter}
\author{Tong Wu$^{a}$}
\ead{wut977@nenu.edu.cn}
\author{Yong Wang$^{b,*}$}
\ead{wangy581@nenu.edu.cn}
\cortext[cor]{Corresponding author.}
\address{$^a$Department of Mathematics, Northeastern University, Shenyang, 110819, China}
\address{$^b$School of Mathematics and Statistics, Northeast Normal University,
Changchun, 130024, China}

\title{The linear perturbation of the metric and the bimetric conformal invarints}
\begin{abstract}
In this paper, we give a method to construct bimetric conformal invarints by the linear metric perturbations and the conformal invarints. And we compute the metric perturbations of the Connes conformal invarints and the conformal Laplacian. As corollaries, some new bimetric conformal invarints on 4-dimensional Riemannian manifolds without boundary are obtained and we get the first order and second order variations of the Connes conformal invarints.
\end{abstract}
\end{frontmatter}
\section{Introduction}
In recent years, the application of perturbation theory in fundamental physical problems has attracted widespread attention. Perturbation theory is a mathematical tool used to solve complex problems that are difficult or impossible to solve directly. The core idea of this method is to decompose the complex problem into two parts: a solvable exact problem and a small perturbation. By introducing a small perturbation on the basis of the exact solution, perturbation theory can gradually unfold the solution, thereby providing an approximate but sufficiently precise result\cite{B1}.

Within the framework of General Relativity, perturbation theory holds indispensable value for dissecting and comprehending dynamic processes, and it has now become an efficient and logically consistent supplementary tool for complete numerical relativity simulations \cite{S1}. Perturbation theory is particularly applied to investigate the stability characteristics of various solutions, including black hole spacetimes\cite{Ch2}, cosmological solutions\cite{Mi2}, and critical point solutions\cite{G2}, among others. Furthermore, this theory assists us in examining potential gauge instabilities\cite{K1}, constraint violations\cite{G3}, and other forms of instability within the different formulations of the Einstein equations implemented in numerical relativity, as numerical errors can essentially be regarded as distortions of the solution being computed. Some metric perturbations of general relativity have been also studied \cite{G1,O1,H1}. In \cite{G1}, Garc\'{\i}a-Bellido J and Wands D studied the metric perturbations produced during inflation in models with two scalar fields evolving simultaneously. In \cite{O1}, Okounkova M, Scheel M A and Teukolsky S A calculated the NC corrections to the metric perturbations around the Schwarzschild blak hole. In \cite{H1}, Herceg N, Juri T, Samsarov A, et al. evolved the leading-order metric perturbation in order-reduced dynamical Chern-Simons gravity. In the previous articles, it was more about the metric perturbations in a physical environment. In this paper, we introduce the metric perturbation into the mathematical environment and further study it.

On the other hand, Connes discussed the conformal aspect of noncommutative
geometry and computed the value of the differential form $\Omega(f_1,f_2)$ at $x$, in terms of $f_1,f_2$ and the conformal structure of $\Sigma$. Ugalde studied the differential form $\Omega_{n,S}$ for the case $(B,S)=(\mathcal{H},F)$ and constructed a conformal invariant differential operator of order the dimension of the manifold for an even dimensional, compact, conformal manifold without boundary in \cite{UW2} and \cite{UW1}. In \cite{B2}, Bochniak A discussed a class of doubled geometry models with diagonal metrics and formulated a hypothesis that supports treating them as modified bimetric gravity theories. In \cite{B3}, Bochniak A and Sitarz A proposed an effective gravity action that couples the two metrics in a similar manner as in bimetric theory of gravity and analyzed whether standard solutions with identical metrics are stable under small perturbations. Therefore, it is a natural question to constrct bimetric conformal invariants. {\bf The main innovation} in this paper is to construct bimetric conformal invariants and study the Connes conformal invarints and the conformal laplacian by the metric perturbations.

\indent The organization of this paper is as follows. In Section \ref{section:2}, we give basic concepts about the conformal invariant, and construct bimetric conformal invariants by the metric perturbations on compact manifolds. In Section \ref{section:3}, we compute the metric perturbations of the Connes conformal invarints and the conformal laplacian in natural frame. In Section \ref{section:4}, the first order and second order variations of the Connes conformal invarints are obtained. Finally, some complex coefficients are listed in the Appendix.
\section{Bimetric Conformal Invariants on Compact Manifolds}
\label{section:2}
In this section, we introduce basic concepts about the conformal invariants and define the bimetric conformal invariants by the metric perturbations on compact manifolds.
Consider an n-dimensional compact Riemannian manifold $M$ and $g$ is Riemannian metric on the tangent bundle $TM$ of $M$, and let $e_1,e_2,\cdot\cdot\cdot,e_n$ be local orthonormal basis on $M$, $e_1^*,e_2^*,\cdot\cdot\cdot,e_n^*$ be the dual basis. We begin with the following definition
\begin{defn}\label{def1}We call that the function $F_g$ is the conformal invariant, if $F_g=F_{fg},$ for $f\in C^\infty(M)$ and $f>0.$
\end{defn}
Take the linear perturbation of the metric $g$, that is $g=\bar{g}+\varepsilon\bar{\bar{g}}$, $\forall \varepsilon>0,$ where $\bar{g}$ and $\bar{\bar{g}}$ are Riemannian metrics on the tangent bundle $TM$ of $M$. Let $F_g$ be a conformal invariant, and $F_{\bar{g}+\varepsilon\bar{\bar{g}}}=F_{\bar{g}}+F^{1}_{(\bar{g},\bar{\bar{g}})}\varepsilon+F^{2}_{(\bar{g},\bar{\bar{g}})}\varepsilon^2+\cdot\cdot\cdot$ Then we have the following result
\begin{thm}\label{thm1}$F^m_{(\bar{g},\bar{\bar{g}})}$, $m\in Z^+$ are bimetric conformal invariants of $(\bar{g},\bar{\bar{g}}),$ that is $F^m_{(\bar{g},\bar{\bar{g}})}=F^m_{f(\bar{g},\bar{\bar{g}})}.$
\end{thm}
\begin{proof}
By Definition \ref{def1}, we have
\begin{align}\label{a1}
F_{\bar{g}+\varepsilon\bar{\bar{g}}}=F_{f\bar{g}+\varepsilon f\bar{\bar{g}}}.
\end{align}
Taylor expansion on the left and right sides of Eq.(\ref{a1}) with respect to $\varepsilon,$ the following equation is obtained.
\begin{align}\label{a2}
&F_{\bar{g}}+F^{1}_{(\bar{g},\bar{\bar{g}})}\varepsilon+F^{2}_{(\bar{g},\bar{\bar{g}})}\varepsilon^2+\cdot\cdot\cdot=F_{f\bar{g}}+F^{1}_{(f\bar{g},f\bar{\bar{g}})}\varepsilon+F^{2}_{(f\bar{g},f\bar{\bar{g}})}\varepsilon^2+\cdot\cdot\cdot.
\end{align}
Since Eq.(\ref{a2}) holds for $\forall \varepsilon >0$, the corresponding coefficients of $\varepsilon$ on the left and right sides of Eq.(\ref{a1}) are equal. Then $$F^m_{(\bar{g},\bar{\bar{g}})}=F^m_{f(\bar{g},\bar{\bar{g}})}.$$
Therefore, Theorem \ref{thm1} holds.
\end{proof}
From \cite{Co1}, we have the following theorem
\begin{thm}\cite{Co1}Let $M$ be a compact $2m$-dimensional conformal manifold and $F=\frac{d\delta-\delta d}{d\delta+\delta d}$ acting on $\Gamma(\Lambda^m(T^*M)),$ and ${\rm Wres}(f_0[F,f_1][F,f_2])=\int_Mf_0\Omega_n(f_1,f_2),$ then ${\rm Wres}(f_0[F,f_1][F,f_2])$ is a Hochschild cocycle on $C^\infty(M) $ and $\Omega_n(f_1,f_2)$is a conformal invariant.
\end{thm}
Using Eq.(2.33) in \cite{Co1}, we get the following Lemma on 4-dimensional manifold:
\begin{lem}\cite{Co1}
There exists a universal trilinear form in the curvature
$r$ and the covectors $df_1,df_2$ such that, in full generality, one has:
\begin{align}\label{aaaa}
\Omega_4(f_1,f_2)_g&=\Big[\frac{1}{3}r<df_1,df_2>+\Delta<df_1,df_2>+<\nabla df_1,\nabla df_2>-\frac{1}{2}\Delta f_1\Delta f_2\Big]Vol_{M,g}\nonumber\\
&=A_4(f_1,f_2)_gVol_{M,g},
\end{align}
where $\Delta=-\sum_{ij}g^{ij}(x)(\partial_i\partial_j-\sum_k\Gamma^k_{ij}\partial_k)$ is the Laplacian and $\nabla$ denotes the covariant derivative.
\end{lem}
On 4-dimensonal manifold, we know
\begin{align*}
Vol_{M,g}=e_1^*\wedge e_2^*\wedge e_3^*\wedge e_4^*,
\end{align*}
then
\begin{align}\label{a3}
Vol_{M,fg}&=f^2e_1^*\wedge e_2^*\wedge e_3^*\wedge e_4^*\nonumber\\
&=f^2Vol_{M,g}.
\end{align}
Since $\Omega_4(f_1,f_2)_g$ is the Connes conformal invariant, we have
\begin{align*}
A_4(f_1,f_2)_gVol_{M,g}=\Omega_4(f_1,f_2)_g=\Omega_4(f_1,f_2)_{fg}=A_4(f_1,f_2)_{fg}Vol_{M,fg}.
\end{align*}
Using Eq.(\ref{a3}), we get
\begin{align*}
A_4(f_1,f_2)_{fg}=f^{-2}A_4(f_1,f_2)_g.
\end{align*}
It leads to
\begin{align}\label{a4}
A_4(f_1,f_2)_{f(\bar{g}+\varepsilon\bar{\bar{g}})}=f^{-2}A_4(f_1,f_2)_{\bar{g}+\varepsilon\bar{\bar{g}}}.
\end{align}
Taylor expansion on the left and right sides of Eq.(\ref{a4}) with respect to $\varepsilon,$ the following equation is obtained.
\begin{align}\label{a25}
&F_{\bar{g}}+F^{1}_{(\bar{g},\bar{\bar{g}})}\varepsilon+F^{2}_{(\bar{g},\bar{\bar{g}})}\varepsilon^2+\cdot\cdot\cdot=F_{f\bar{g}}+F^{1}_{(f\bar{g},f\bar{\bar{g}})}\varepsilon+F^{2}_{(f\bar{g},f\bar{\bar{g}})}\varepsilon^2+\cdot\cdot\cdot.
\end{align}
Then
\begin{align*}
A_4^m(f_1,f_2)_{f(\bar{g},\bar{\bar{g}})}=f^{-2}A_4^m(f_1,f_2)_{\bar{g},\bar{\bar{g}}},~~~m\in Z.
\end{align*}
Denote
\begin{align*}
&A_4^m(f_1,f_2)_{(\bar{g},\bar{\bar{g}})}Vol_{M,\bar{g}}:=B^m(f_1,f_2)_{(\bar{g},\bar{\bar{g}})};\nonumber\\
&A_4^m(f_1,f_2)_{f(\bar{g},\bar{\bar{g}})}Vol_{M,f\bar{g}}:=B^m(f_1,f_2)_{f(\bar{g},\bar{\bar{g}})}.
\end{align*}
Then we obtain the following theorem
\begin{thm}\label{thm2}$B^m(f_1,f_2)_{(\bar{g},\bar{\bar{g}})}$, $m\in Z^+$ are bimetric conformal invariants of 4-dimensional compact oriented manifold without boundary.
\end{thm}
\begin{rem}\label{cor1}The metric Pertubations $g=g^0+\varepsilon g^1+\cdot\cdot\cdot+\varepsilon^ng^n$ can generate $(n+1)$ metric conformal invariants of metric.
\end{rem}
\section{The metric Perturbations of the Connes conformal invariants and the conformal laplacian}
\label{section:3}
In this section, we compute the metric perturbations of  the Connes conformal invariant $\Omega_4(f_1,f_2)_{\bar{g}+\varepsilon\bar{\bar{g}}}$ and the conformal laplacian generated by $\varepsilon.$ Suppose that $\partial_{i}$ is a natural local frame on $TM$
and $(g^{ij})_{1\leq i,j\leq n}$ is the inverse matrix associated to the metric matrix  $(g_{ij})_{1\leq i,j\leq n}$ on $M$. Then, we have
\begin{align}\label{a5}
g_{ij}=\bar{g}_{ij}+\varepsilon\bar{\bar{g}}_{ij}.
\end{align}
The expansion for the inverse metric can be obtained by iteration of the identity, $$g^{\mu\nu}=\bar{g}^{\mu\nu}-\bar{g}^{\mu\lambda}(g_{\lambda\sigma}-\bar{g}_{\lambda\sigma})g^{\sigma\nu}.$$
Because $g^{\mu\nu}$ is a power series of $\varepsilon,$ suppose $g^{\mu\nu}=g^{\mu\nu}[0]+g^{\mu\nu}[1]\varepsilon+g^{\mu\nu}[2]\varepsilon^2+0(\varepsilon^3).$ Obviously, the constant term of $g^{\mu\nu}$ is $\bar{g}^{\mu\nu}.$ That is, $g^{\mu\nu}[0]=\bar{g}^{\mu\nu}$. By Eq.(\ref{a5}), we get
\begin{align*}
g^{\mu\nu}[1]&=\Big(\bar{g}^{\mu\nu}-\bar{g}^{\mu\lambda}(\varepsilon \bar{\bar{g}}_{\lambda\sigma})g^{\sigma\nu}\Big)[1]\nonumber\\
&=-\bar{g}^{\mu\lambda}\bar{\bar{g}}_{\lambda\sigma}\bar{g}^{\sigma\nu},
\end{align*}
and
\begin{align*}
g^{\mu\nu}[2]&=\Big(\bar{g}^{\mu\nu}-\bar{g}^{\mu\lambda}(\varepsilon \bar{\bar{g}}_{\lambda\sigma})g^{\sigma\nu}\Big)[2]\nonumber\\
&=-\bar{g}^{\mu\lambda}\bar{\bar{g}}_{\lambda\sigma}g^{\sigma\nu}[1]\nonumber\\
&=\bar{g}^{\mu\lambda}\bar{\bar{g}}_{\lambda\sigma}g^{\sigma\alpha}\bar{\bar{g}}_{\alpha\beta}\bar{g}^{\beta\nu}.
\end{align*}
Therefore, we obtain
\begin{align}\label{a6}
g^{\mu\nu}=\bar{g}^{\mu\nu}-\bar{g}^{\mu\lambda}\bar{\bar{g}}_{\lambda\sigma}\bar{g}^{\sigma\nu}\varepsilon+\bar{g}^{\mu\lambda}\bar{\bar{g}}_{\lambda\sigma}g^{\sigma\alpha}\bar{\bar{g}}_{\alpha\beta}\bar{g}^{\beta\nu}\varepsilon^2+O(\varepsilon^3).
\end{align}

Let $\nabla^L$ be the Levi-Civita connection about $g$ on $M$, then $\nabla^L_{\partial_i}\partial_j=\sum_{k}\Gamma_{ ij}^{k}\partial_k,$ where $\Gamma_{ ij}^{k}$ is the  Christoffel coefficient of $\nabla^{L}$. And by definition
\begin{align}\label{a7}
\Gamma_{ij}^{k}=\frac{1}{2}g^{kl}\Big(\frac{\partial g_{il}}{\partial_{x_j}}+\frac{\partial g_{jl}}{\partial_{x_i}}-\frac{\partial g_{ij}}{\partial_{x_l}}\Big).
\end{align}
Put the result of Eq.(\ref{a6}) into Eq.(\ref{a7}), we get $\Gamma_{ij}^{k}$ is a power series of $\varepsilon.$
\begin{align}\label{a8}
\Gamma_{ij}^{k}&=\bar{\Gamma}_{ij}^{k}+\frac{1}{2}\Big[\bar{g}^{kl}\Big(\frac{\partial  \bar{\bar{g}}_{il}}{\partial_{x_j}}+\frac{\partial \bar{\bar{g}}_{jl}}{\partial_{x_i}}-\frac{\partial \bar{\bar{g}}_{ij}}{\partial_{x_l}}\Big)-\bar{g}^{k\lambda}\bar{\bar{g}}_{\lambda\sigma}\bar{g}^{\sigma l}\Big(\frac{\partial \bar{g}_{il}}{\partial_{x_j}}+\frac{\partial \bar{g}_{jl}}{\partial_{x_i}}-\frac{\partial \bar{g}_{ij}}{\partial_{x_l}}\Big)\Big]\varepsilon\nonumber\\
&+\frac{1}{2}\Big[-\bar{g}^{kl}\bar{\bar{g}}_{\lambda\sigma}\bar{g}^{\sigma l}\Big(\frac{\partial  \bar{\bar{g}}_{il}}{\partial_{x_j}}+\frac{\partial \bar{\bar{g}}_{jl}}{\partial_{x_i}}-\frac{\partial \bar{\bar{g}}_{ij}}{\partial_{x_l}}\Big)+\bar{g}^{k\lambda}\bar{\bar{g}}_{\lambda\sigma}\bar{g}^{\sigma\alpha}\bar{\bar{g}}_{\alpha\beta}\bar{g}^{\beta l}\Big(\frac{\partial \bar{g}_{il}}{\partial_{x_j}}+\frac{\partial \bar{g}_{jl}}{\partial_{x_i}}-\frac{\partial \bar{g}_{ij}}{\partial_{x_l}}\Big)\Big]\varepsilon^2+O(\varepsilon^3).
\end{align}
\subsection{The Connes conformal invariants}
Firstly, we compute scalar curvature curvature $r$ by the above result. Since the curvature tensor $R^L$ is given by the commutator of $\nabla^L$, that is
\begin{align}\label{a9}
R^L(X,Y)Z=(\nabla^L_X\nabla^L_Y-\nabla^L_Y\nabla^L_X-\nabla_{[X,Y]})Z,
\end{align}
and
\begin{align}\label{a10}
-<R^L(X,Y)Z,W>:=R^L(X,Y,Z,W).
\end{align}
In natural local frame on $TM$, we have
\begin{align*}
r_{\bar{g}+\varepsilon\bar{\bar{g}}}=g^{jl}g^{km}R^L(\partial_i,\partial_k,\partial_l,\partial_m).
\end{align*}
By Eq.(\ref{a9}) and Eq.(\ref{a10}), we get
\begin{align*}
R^L(\partial_j,\partial_k,\partial_l,\partial_m)&=-<[\nabla^L_{\partial_j}\nabla^L_{\partial_k}-\nabla^L_{\partial_k}\nabla^L_{\partial_j}]\partial_l,\partial_m>\nonumber\\
&=-<[\nabla^L_{\partial_j}(\Gamma^\alpha_{kl}\partial_\alpha)-\nabla^L_{\partial_k}(\Gamma^\alpha_{jl}\partial_\alpha),\partial_m>\nonumber\\
&=-<[\partial_j(\Gamma^\alpha_{kl})\partial_\alpha+\Gamma^\alpha_{kl}\Gamma^\beta_{j\alpha}\partial_\beta-\partial_k(\Gamma^\alpha_{jl})\partial_\alpha-\Gamma^\alpha_{jl}\Gamma^\beta_{k\alpha}\partial_\beta,\partial_m>\nonumber\\
&=-\partial_j(\Gamma^\alpha_{kl})g_{\alpha m}-\Gamma^\alpha_{kl}\Gamma^\beta_{j\alpha}g_{\beta m}+\partial_k(\Gamma^\alpha_{jl})g_{\alpha m}+\Gamma^\alpha_{jl}\Gamma^\beta_{k\alpha}g_{\beta m}.
\end{align*}
 To compute $r$, we can use the identity $g^{ab}g_{bc}=\delta^a_c,$ then
\begin{align}\label{a11}
r&=g^{jl}g^{km}[\partial_k(\Gamma^\alpha_{jl})g_{\alpha m}+\Gamma^\alpha_{jl}\Gamma^\beta_{k\alpha}g_{\beta m}-\partial_j(\Gamma^\alpha_{kl})g_{\alpha m}-\Gamma^\alpha_{kl}\Gamma^\beta_{j\alpha}g_{\beta m}]\nonumber\\
&=g^{jl}\partial_k(\Gamma^\alpha_{jl})\delta_{\alpha}^k+g^{jl}\Gamma^\alpha_{jl}\Gamma^\beta_{k\alpha}\delta^k_\beta-g^{jl}\partial_j(\Gamma^\alpha_{kl})\delta_{\alpha}^k-g^{jl}\Gamma^\alpha_{kl}\Gamma^\beta_{j\alpha}\delta^k_\beta\nonumber\\
&=g^{jl}\partial_k(\Gamma^k_{jl})+g^{jl}\Gamma^\alpha_{jl}\Gamma^k_{k\alpha}-g^{jl}\partial_j(\Gamma^k_{kl})-g^{jl}\Gamma^\alpha_{kl}\Gamma^k_{j\alpha}.
\end{align}
Put the result of Eq.(\ref{a6}) and Eq.(\ref{a8}) into Eq.(\ref{a11}), we have that $r$ is a power series of $\varepsilon.$
\begin{align*}
r=r_0+r_1\varepsilon+r_2\varepsilon^2+O(\varepsilon^3),
\end{align*}
where $r_0,r_1,r_2$ are the coefficients of $\varepsilon^0,\varepsilon,\varepsilon^2$ in the scalar curvature $r$, which are shown in the Appendix  Eq.(\ref{b1}).

Likewise, by Eq.(\ref{a6}) and Eq.(\ref{a8}), we can expand $<df_1,df_2>,\Delta<df_1,df_2>,<\nabla df_1,\nabla df_2>,\Delta f_1\Delta f_2$ into a power series of $\varepsilon$ respectively.
\begin{align*}
<df_1,df_2>&=<\sum_j\frac{\partial f_1}{\partial x_j}dx_j,\sum_l\frac{\partial f_2}{\partial x_l}dx_l>\nonumber\\
&=\sum_{jl}\frac{\partial f_1}{\partial x_j}\frac{\partial f_2}{\partial x_l}g^{jl}\nonumber\\
&=\sum_{jl}\frac{\partial f_1}{\partial x_j}\frac{\partial f_2}{\partial x_l}[\bar{g}^{jl}-\bar{g}^{j\lambda}\bar{\bar{g}}_{\lambda\sigma}\bar{g}^{\sigma l}\varepsilon+\bar{g}^{j\lambda}\bar{\bar{g}}_{\lambda\sigma}g^{\sigma\alpha}\bar{\bar{g}}_{\alpha\beta}\bar{g}^{\beta l}\varepsilon^2+O(\varepsilon^3)]\nonumber\\
&=t_0+t_1\varepsilon+t_2\varepsilon^2+O(\varepsilon^3),
\end{align*}
where
\begin{align*}
&t_0=\sum_{jl}\frac{\partial f_1}{\partial x_j}\frac{\partial f_2}{\partial x_l}\bar{g}^{jl};\nonumber\\
&t_1=-\sum_{jl}\frac{\partial f_1}{\partial x_j}\frac{\partial f_2}{\partial x_l}\bar{g}^{j\lambda}\bar{\bar{g}}_{\lambda\sigma}\bar{g}^{\sigma l};\nonumber\\
&t_2=\sum_{jl}\frac{\partial f_1}{\partial x_j}\frac{\partial f_2}{\partial x_l}\bar{g}^{j\lambda}\bar{\bar{g}}_{\lambda\sigma}g^{\sigma\alpha}\bar{\bar{g}}_{\alpha\beta}\bar{g}^{\beta l}.
\end{align*}
\begin{align*}
\Delta<df_1,df_2>&=\Delta\Big(\sum_{jl}\frac{\partial f_1}{\partial x_j}\frac{\partial f_2}{\partial x_l}g^{jl}\Big)\nonumber\\
&=\Big[-\sum_{\alpha\beta}g^{\alpha\beta}(x)(\partial_\alpha\partial_\beta-\sum_k\Gamma^k_{\alpha\beta}\partial_k)\Big]\Big(\sum_{jl}\frac{\partial f_1}{\partial x_j}\frac{\partial f_2}{\partial x_l}g^{jl}\Big)\nonumber\\
&=a_0+a_1\varepsilon+a_2\varepsilon^2+O(\varepsilon^3),
\end{align*}
where $a_0,a_1,a_2$ are the coefficients of $\varepsilon^0,\varepsilon,\varepsilon^2$ in $\Delta<df_1,df_2>$, which are shown in the Appendix Eq.(\ref{b2}).

\begin{align*}
\nabla df_1&=\nabla\Big(\sum_l\frac{\partial f_1}{\partial x_l}dx_l\Big)\nonumber\\
&=\sum_ld(\frac{\partial f_1}{\partial x_l})\otimes dx_l+\sum_l\frac{\partial f_1}{\partial x_l}\nabla(dx_l).
\end{align*}
By $<\nabla_{\partial_{x_k}}\partial_{x_\alpha},dx_l>+<\partial_{x_\alpha},\nabla_{\partial_{x_k}}dx_l>=\partial_{x_k}(<\partial_{x_\alpha},dx_l>)=0,$ we have
\begin{align*}
\nabla (dx_l)&=\sum_kdx^k\otimes\nabla_{\partial_{x_k}}dx_l\nonumber\\
&=-\sum_{k\beta}\Gamma^l_{k\beta}dx_k\otimes dx_\beta.
\end{align*}
This leads to
\begin{align*}
\nabla df_1&=\sum_{l\beta}\frac{\partial^2f_1}{\partial x_l\partial x_\beta}dx_\beta\otimes dx_l-\sum_{lk\beta}\frac{\partial f_1}{\partial x_l}\Gamma^l_{k\beta}dx_k\otimes dx_\beta\nonumber\\
&=\sum_{l\beta}\Big(\frac{\partial^2f_1}{\partial x_l\partial x_\beta}-\sum_k\frac{\partial f_1}{\partial x_k}\Gamma^k_{\beta l}\Big)dx_\beta\otimes dx_l.
\end{align*}
Moreover, by $<dx_\beta\otimes dx_l,dx_p\otimes dx_q>=<dx_\beta,dx_p><dx_l,dx_q>=g^{\beta p}g^{lq},$ then we obtain
\begin{align*}
<\nabla df_1,\nabla df_2>&=<\sum_{l\beta}\Big(\frac{\partial^2f_1}{\partial x_l\partial x_\beta}-\sum_k\frac{\partial f_1}{\partial x_k}\Big)\Gamma^k_{\beta l}dx_\beta\otimes dx_l,\sum_{pq}\Big(\frac{\partial^2f_2}{\partial x_p\partial x_q}-\sum_m\frac{\partial f_2}{\partial x_m}\Big)\Gamma^m_{pq}dx_p\otimes dx_q>\nonumber\\
&=\sum_{l\beta}\sum_{pq}\Big(\frac{\partial^2f_1}{\partial x_l\partial x_\beta}-\sum_k\frac{\partial f_1}{\partial x_k}\Gamma^k_{\beta l}\Big)\Big(\frac{\partial^2f_2}{\partial x_p\partial x_q}-\sum_m\frac{\partial f_2}{\partial x_m}\Gamma^m_{pq}\Big)g^{\beta p}g^{lq}\nonumber\\
&=b_0+b_1\varepsilon+b_2\varepsilon^2+O(\varepsilon^3),
\end{align*}
where $b_0,b_1,b_2$ are the coefficients of $\varepsilon^0,\varepsilon,\varepsilon^2$ in $<\nabla f_1,\nabla f_2>$, which are shown in the Appendix Eq.(\ref{b3}).

\begin{align*}
\Delta f_1\Delta f_2&=\Big[\sum_{\alpha\beta}g^{\alpha\beta}(x)\Big(\partial_\alpha\partial_\beta-\sum_l\Gamma^l_{\alpha\beta}\partial_l\Big)(f_1)\Big]\Big[\sum_{pq}g^{pq}(x)\Big(\partial_p\partial_q-\sum_r\Gamma^r_{pq}\partial_r\Big)(f_2)\Big]\nonumber\\
&=d_0+d_1\varepsilon+d_2\varepsilon^2+O(\varepsilon^3),
\end{align*}
where $d_0,d_1,d_2$ are the coefficients of $\varepsilon^0,\varepsilon,\varepsilon^2$ in $\Delta f_1\Delta f_2$, which are shown in the Appendix Eq.(\ref{b4}).

By Eq.(\ref{aaaa}) and above power series of $\varepsilon$, $A_4(f_1,f_2)_{\bar{g}+\varepsilon \bar{\bar{g}}}$ can be rearranged as follows:
\begin{align}\label{a12}
A_4(f_1,f_2)_{\bar{g}+\varepsilon \bar{\bar{g}}}&=\frac{1}{3}r<df_1,df_2>+\Delta<df_1,df_2>+<\nabla df_1,\nabla df_2>-\frac{1}{2}\Delta f_1\Delta f_2\nonumber\\
&=\frac{1}{3}r_0t_0+a_0+b_0-\frac{1}{2}d_0+\Big(\frac{1}{3}(r_1t_0+r_0t_1)+a_1+b_1-\frac{1}{2}d_1\Big)\varepsilon\nonumber\\
&+\Big(\frac{1}{3}(r_2t_0+r_0t_2+r_1t_1)+a_2+b_2-\frac{1}{2}d_2\Big)\varepsilon^2+O(\varepsilon^3).
\end{align}
It follows that we obtain the metric Pertubations of the Connes conformal invariants.
\begin{align}\label{aaaa}
\Omega_4(f_1,f_2)_g&=A_4(f_1,f_2)_gVol_{M,g}\nonumber\\
&=\Big[\frac{1}{3}r_0t_0+a_0+b_0-\frac{1}{2}d_0+\Big(\frac{1}{3}(r_1t_0+r_0t_1)+a_1+b_1-\frac{1}{2}d_1\Big)\varepsilon\nonumber\\
&+\Big(\frac{1}{3}(r_2t_0+r_0t_2+r_1t_1)+a_2+b_2-\frac{1}{2}d_2\Big)\varepsilon^2+O(\varepsilon^3)\Big]Vol_{M,g}.
\end{align}
\subsection{The conformal laplacian}
From \cite{Mi1}, we have the conformal laplacian in the following way:
\begin{align*}
\widetilde{\Delta}=\Delta+\frac{n-2}{4(n-1)}r.
\end{align*}
By Eq.(\ref{a6}) and Eq.(\ref{a8}), we get the metric perturbations of the laplacian.
\begin{align*}
\Delta&=-\sum_{ij}g^{ij}\Big(\partial_i\partial_j-\Gamma^k_{ij}\partial_k\Big)\nonumber\\
&=-\sum_{ij}\bar{g}^{ij}\Big(\partial_i\partial_j-\bar{\Gamma}^k_{ij}\partial_k\Big)+\Big\{\sum_{ij}\bar{g}^{j\lambda}\bar{\bar{g}}_{\lambda\sigma}\bar{g}^{\sigma j}\Big(\partial_i\partial_j-\bar{\Gamma}^k_{ij}\partial_k\Big)+\frac{1}{2}\sum_{ij}\bar{g}^{ij}\Big[\bar{g}^{kl}\Big(\frac{\partial \bar{\bar{g}}^{il}}{\partial{x_j}}+\frac{\partial \bar{\bar{g}}^{jl}}{\partial{x_i}}-\frac{\partial \bar{\bar{g}}^{ij}}{\partial{x_l}}\Big)\nonumber\\
&-\bar{g}^{kl}\bar{\bar{g}}_{\lambda\sigma}\bar{g}^{\sigma l}\Big(\frac{\partial \bar{g}^{il}}{\partial{x_j}}+\frac{\partial \bar{g}^{jl}}{\partial{x_i}}-\frac{\partial \bar{g}^{ij}}{\partial{x_l}}\Big)\Big]\Big\}\varepsilon+\Big\{-\sum_{ij}\bar{g}^{j\lambda}\bar{\bar{g}}_{\lambda\sigma}\bar{g}^{\sigma l}\bar{\bar{g}}_{l\beta}\bar{g}^{\beta j}\Big(\partial_i\partial_j-\bar{\Gamma}^k_{ij}\partial_k\Big)+\frac{1}{2}\sum_{ij}\bar{g}^{ij}\nonumber\\
&\Big[-\bar{g}^{k\lambda}\bar{\bar{g}}_{\lambda\sigma}\bar{g}^{\sigma l}\Big(\frac{\partial \bar{\bar{g}}^{il}}{\partial{x_j}}+\frac{\partial \bar{\bar{g}}^{jl}}{\partial{x_i}}-\frac{\partial \bar{\bar{g}}^{ij}}{\partial{x_l}}\Big)+\bar{g}^{kl}\bar{\bar{g}}_{\lambda\sigma}\bar{g}^{\sigma \alpha}\bar{\bar{g}}_{\alpha\beta}\bar{g}^{\beta l}\Big(\frac{\partial \bar{g}^{il}}{\partial_{x_j}}+\frac{\partial \bar{g}^{jl}}{\partial{x_i}}-\frac{\partial \bar{g}^{ij}}{\partial{x_l}}\Big)\Big]-\frac{1}{2}\sum_{ij}\bar{g}^{i\lambda}\bar{\bar{g}}_{\lambda\sigma}\bar{g}^{\sigma j}\nonumber\\
&\Big[\bar{g}^{kl}\Big(\frac{\partial \bar{\bar{g}}^{il}}{\partial{x_j}}+\frac{\partial \bar{\bar{g}}^{jl}}{\partial{x_i}}-\frac{\partial \bar{\bar{g}}^{ij}}{\partial{x_l}}\Big)-\bar{g}^{kl}\bar{\bar{g}}_{\lambda\sigma}\bar{g}^{\sigma l}\Big(\frac{\partial \bar{g}^{il}}{\partial{x_j}}+\frac{\partial \bar{g}^{jl}}{\partial{x_i}}-\frac{\partial \bar{g}^{ij}}{\partial{x_l}}\Big)\Big]\partial_k\Big\}\varepsilon^2+O(\varepsilon^3)\nonumber\\
&=p_0+p_1\varepsilon+p_2\varepsilon^2+O(\varepsilon^3),
\end{align*}
where
\begin{align*}
p_0&=-\sum_{ij}\bar{g}^{ij}(\partial_i\partial_j-\bar{\Gamma}^k_{ij}\partial_k);\nonumber\\
p_1&=\sum_{ij}\bar{g}^{j\lambda}\bar{\bar{g}}_{\lambda\sigma}\bar{g}^{\sigma j}\Big(\partial_i\partial_j-\bar{\Gamma}^k_{ij}\partial_k\Big)+\frac{1}{2}\sum_{ij}\bar{g}^{ij}\Big[\bar{g}^{kl}\Big(\frac{\partial \bar{\bar{g}}^{il}}{\partial{x_j}}+\frac{\partial \bar{\bar{g}}^{jl}}{\partial{x_i}}-\frac{\partial \bar{\bar{g}}^{ij}}{\partial{x_l}}\Big)-\bar{g}^{kl}\bar{\bar{g}}_{\lambda\sigma}\bar{g}^{\sigma l}\Big(\frac{\partial \bar{g}^{il}}{\partial{x_j}}+\frac{\partial \bar{g}^{jl}}{\partial{x_i}}-\frac{\partial \bar{g}^{ij}}{\partial{x_l}}\Big)\Big];\nonumber\\
p_2&=-\sum_{ij}\bar{g}^{j\lambda}\bar{\bar{g}}_{\lambda\sigma}\bar{g}^{\sigma l}\bar{\bar{g}}_{l\beta}\bar{g}^{\beta j}\Big(\partial_i\partial_j-\bar{\Gamma}^k_{ij}\partial_k\Big)+\frac{1}{2}\sum_{ij}\bar{g}^{ij}\Big[-\bar{g}^{k\lambda}\bar{\bar{g}}_{\lambda\sigma}\bar{g}^{\sigma l}\Big(\frac{\partial \bar{\bar{g}}^{il}}{\partial{x_j}}+\frac{\partial \bar{\bar{g}}^{jl}}{\partial{x_i}}-\frac{\partial \bar{\bar{g}}^{ij}}{\partial{x_l}}\Big)+\bar{g}^{kl}\bar{\bar{g}}_{\lambda\sigma}\bar{g}^{\sigma \alpha}\bar{\bar{g}}_{\alpha\beta}\bar{g}^{\beta l}\nonumber\\
&\Big(\frac{\partial \bar{g}^{il}}{\partial{x_j}}+\frac{\partial \bar{g}^{jl}}{\partial{x_i}}-\frac{\partial \bar{g}^{ij}}{\partial{x_l}}\Big)\Big]-\frac{1}{2}\sum_{ij}\bar{g}^{i\lambda}\bar{\bar{g}}_{\lambda\sigma}\bar{g}^{\sigma j}\Big[\bar{g}^{kl}\Big(\frac{\partial \bar{\bar{g}}^{il}}{\partial{x_j}}+\frac{\partial \bar{\bar{g}}^{jl}}{\partial{x_i}}-\frac{\partial \bar{\bar{g}}^{ij}}{\partial{x_l}}\Big)-\bar{g}^{kl}\bar{\bar{g}}_{\lambda\sigma}\bar{g}^{\sigma l}\Big(\frac{\partial \bar{g}^{il}}{\partial{x_j}}+\frac{\partial \bar{g}^{jl}}{\partial{x_i}}-\frac{\partial \bar{g}^{ij}}{\partial{x_l}}\Big)\Big]\partial_k.
\end{align*}
Finally,  the metric perturbations of the conformal laplacian is given.
\begin{align}\label{lll}
\widetilde{\Delta}=p_0+\frac{n-2}{4(n-1)}r_0+\Big(p_1+\frac{n-2}{4(n-1)}r_1\Big)\varepsilon+\Big(p_2+\frac{n-2}{4(n-1)}r_2\Big)\varepsilon^2+O(\varepsilon^3).
\end{align}
\section{The first order and second order variations of the Connes conformal invarints}
\label{section:4}
In this section, we consider the first order and second order variations of the spetral functional about the Connes conformal invariants and the conformal laplacian. Firstly, from \cite{D1}, we have
\begin{align}\label{a13}
{\rm Wres}(f_0[F_\varepsilon,f_1][F_\varepsilon,f_2])=\int_MA_4(f_1,f_2)_gdVol_{M,g},
\end{align}
where $dVol_{M,g}$ denotes the volume element of manifold $M$ and from (20) in \cite{Ch1}, we have
\begin{align*}
dVol_{M,g}=\sqrt{det[\bar{g}+\varepsilon \bar{\bar{g}}]}dx_1\wedge dx_2\wedge\cdot\cdot\cdot\wedge dx_n.
\end{align*}
Moreover
\begin{align*}
\sqrt{det[\bar{g}+\varepsilon \bar{\bar{g}}]}=\sqrt{det[\bar{g}]}\sqrt{det(I+\varepsilon [\bar{g}]^{-1}[\bar{\bar{g}}])}.
\end{align*}
Write $G=[\bar{g}]^{-1}[\bar{\bar{g}}],$ then $G_{ij}=([\bar{g}]^{-1}[\bar{\bar{g}}])_{ij}.$ Further, through Taylor expansion, the following equation is obtained.
\begin{align*}
\sqrt{det(I+\varepsilon G)}=1+c_1\varepsilon+c_2\varepsilon^2+O(\varepsilon^3).
\end{align*}
And
\begin{align*}
det(I+\varepsilon G)&=[1+c_1\varepsilon+c_2\varepsilon^2+O(\varepsilon^3)]^2\nonumber\\
&=1+2c_1\varepsilon+(2c_2+c_1^2)\varepsilon^2+O(\varepsilon^3).
\end{align*}
Since
\begin{equation*}
det(I+\varepsilon G)=
det\left(
  \begin{array}{ccccc}
   1+\varepsilon G_{11} &\varepsilon G_{12} &\cdots& \varepsilon G_{1n}  \\
   \varepsilon G_{21}  & 1+\varepsilon G_{22} & \cdots&\varepsilon G_{2n}   \\
   \vdots & \vdots & \ddots&G_{(n-1)n} \\
  \varepsilon G_{n1} &\varepsilon G_{n2} & \cdots& 1+\varepsilon G_{nn}\\
  \end{array}
\right).
\end{equation*}
When $\varepsilon=0,$ we get the constant term. In the same way, we can get the coefficients of $\varepsilon$ and $\varepsilon^2$ respectively.
\begin{align*}
2c_1&=G_{11}+\cdot\cdot\cdot+G_{nn};\nonumber\\
2c_2+c_1^2&=\sum_{1\leq j<l\leq n}G_{jj}G_{ll}-\sum_{1\leq j<l\leq n}G_{jl}G_{lj}.
\end{align*}
Then
\begin{align*}
 c_1&=\frac{G_{11}+\cdot\cdot\cdot+G_{nn}}{2};\nonumber\\
 c_2&=\frac{1}{2}\Big(\sum_{1\leq j<l\leq n}G_{jj}G_{ll}-\sum_{1\leq j<l\leq n}G_{jl}G_{lj}\Big)-\frac{1}{8}(G_{11}+\cdot\cdot\cdot+G_{nn})^2.
\end{align*}
Therefore, we get
\begin{align}\label{a14}
dVol_{M,g}&=\sqrt{det[\bar{g}]}[1+c_1\varepsilon+c_2\varepsilon^2+O(\varepsilon^3)]dx_1\wedge dx_2\wedge\cdot\cdot\cdot\wedge dx_n\nonumber\\
&=[1+c_1\varepsilon+c_2\varepsilon^2+O(\varepsilon^3)]dVol_{M,\bar{g}}.
\end{align}
By $dVol_{M,fg}=f^2dVol_{M,g},$ then $(c_j)_{f(\bar{g},\bar{\bar{g}})}=(c_j)_{(\bar{g},\bar{\bar{g}})},~(0\leq j\leq2).$
By Eq.(\ref{a12}), we can define
\begin{align}
A_4(f_1,f_2)_(\bar{g},\bar{\bar{g}}):=A_4^0(f_1,f_2)_(\bar{g},\bar{\bar{g}})+A_4^1(f_1,f_2)_(\bar{g},\bar{\bar{g}})\varepsilon+A_4^2(f_1,f_2)_(\bar{g},\bar{\bar{g}})\varepsilon^2+O(\varepsilon^3),
\end{align}
then
\begin{align}
A^j_4(f_1,f_2)_(\bar{g},\bar{\bar{g}})=f^{-2}A^j_4(f_1,f_2)_(\bar{g},\bar{\bar{g}})~~~~~(0\leq j\leq2).
\end{align}
 It trivially follows that we get the following result.
\begin{thm}
$A^j_4(f_1,f_2)_(\bar{g},\bar{\bar{g}})c_ldVol_{M,\bar{g}},~(0\leq j,l\leq2,~c_0=1)$ gives rise to nine bimetric conformal invariants.
\end{thm}
\begin{rem}
These bimetric conformal invariants are nontrival and constructing bimetric conformal invariants by the conformal invariant is nontrival. Since the conformal invarianti is the sum of several terms about $g_{ij}$ and $g^{ij}$. Obviously, the above bimetric conformal invariants don't depend on the natural frame.
\end{rem}
Moreover, put the result of Eq.(\ref{a12}) and Eq.(\ref{a14}) into Eq.(\ref{a13}), we get
\begin{align}\label{a15}
{\rm Wres}(f_0[F_\varepsilon,f_1][F_\varepsilon,f_2])&=\int_MA_4(f_1,f_2)_gdVol_{M,g}\nonumber\\
&=\int_M\Big\{\frac{1}{3}r_0t_0+a_0+b_0-\frac{1}{2}d_0+\Big[\frac{1}{3}(r_1t_0+r_0t_1)+a_1+b_1-\frac{1}{2}d_1+\Big(\frac{1}{3}r_0t_0+a_0\nonumber\\
&+b_0-\frac{1}{2}d_0\Big)c_1\Big]\varepsilon+\Big[\frac{1}{3}(r_2t_0+r_0t_2+r_1t_1)+a_2+b_2-\frac{1}{2}d_2+\Big(\frac{1}{3}r_0t_0+a_0+b_0\nonumber\\
&-\frac{1}{2}d_0\Big)c_2+\Big(\frac{1}{3}(r_1t_0+r_0t_1)+a_1+b_1-\frac{1}{2}d_1\Big)c_1\Big]\varepsilon^2+O(\varepsilon^3)\Big\}dVol_{M,\bar{g}}.
\end{align}
 Therefore, using identifications and notation as above, we get the following theorem.
\begin{thm}\label{thm3}Let $M$ be a 4-dimensional manifold, $g=\bar{g}+\varepsilon\bar{\bar{g}}$ is Riemannian metric on the tangent bundle $TM$ of $M$, then the first order and second order variations of the spetral functional about the Connes conformal invariants are obtained.
\begin{align*}
\frac{d}{d\varepsilon}{\rm Wres}(f_0[F_\varepsilon,f_1][F_\varepsilon,f_2])|_{\varepsilon=0}&=\int_M\frac{d}{d\varepsilon}A_4(f_1,f_2)_gdVol_{M,g}|_{\varepsilon=0}\nonumber\\
&=\int_M\Big[\frac{1}{3}(r_1t_0+r_0t_1)+a_1+b_1-\frac{1}{2}d_1+\Big(\frac{1}{3}r_0t_0+a_0+b_0-\frac{1}{2}d_0\Big)c_1\nonumber\\
&+\Big(\frac{1}{3}(r_1t_0+r_0t_1)+a_1+b_1-\frac{1}{2}d_1\Big)c_1\Big]dVol_{M,\bar{g}};
\end{align*}
\begin{align*}
\frac{d^2}{d\varepsilon^2}{\rm Wres}(f_0[F_\varepsilon,f_1][F_\varepsilon,f_2])|_{\varepsilon=0}&=\int_M\frac{d^2}{d^2\varepsilon}A_4(f_1,f_2)_gdVol_{M,g}|_{\varepsilon=0}\nonumber\\
&=2\int_M\Big[\frac{1}{3}(r_2t_0+r_0t_2+r_1t_1)+a_2+b_2-\frac{1}{2}d_2+\Big(\frac{1}{3}r_0t_0+a_0+b_0-\frac{1}{2}d_0\Big)c_2\nonumber\\
&+\Big(\frac{1}{3}(r_1t_0+r_0t_1)+a_1+b_1-\frac{1}{2}d_1\Big)c_1\Big]dVol_{M,\bar{g}}.
\end{align*}
Moreover,$\frac{d}{d\varepsilon}{\rm Wres}(f_0[F_\varepsilon,f_1][F_\varepsilon,f_2])|_{\varepsilon=0}$ and $\frac{d^2}{d\varepsilon^2}{\rm Wres}(f_0[F_\varepsilon,f_1][F_\varepsilon,f_2])|_{\varepsilon=0}$ are two Hochschild cocyles on $C^\infty(M)$.
\end{thm}
Combine Eq.(\ref{lll}) and Eq.(\ref{a14}), we get the following theorem.
\begin{thm}\label{thm3}Let $M$ be an n-dimensional manifold, $g=\bar{g}+\varepsilon\bar{\bar{g}}$ is Riemannian metric on the tangent bundle $TM$ of $M$, then the first order and second order variations of the conformal laplacian are obtained.
\begin{align*}
\frac{d}{d\varepsilon}\widetilde{\Delta}|_{\varepsilon=0}&=p_1+\frac{n-2}{4(n-1)}r_1;
\end{align*}
\begin{align*}
\frac{d^2}{d\varepsilon^2}\widetilde{\Delta}|_{\varepsilon=0}&=2\Big(p_2+\frac{n-2}{4(n-1)}r_2\Big).
\end{align*}
Moreover
\begin{align}
f^{\frac{n+2}{4}}\Big(p_1+\frac{n-2}{4(n-1)}r_1\Big)_{f(\bar{g},\bar{\bar{g}})}=\Big(p_1+\frac{n-2}{4(n-1)}r_1\Big)_{(\bar{g},\bar{\bar{g}})}f^{\frac{n-2}{4}},
\end{align}
and \begin{align}
f^{\frac{n+2}{4}}\Big(p_2+\frac{n-2}{4(n-1)}r_2\Big)_{f(\bar{g},\bar{\bar{g}})}=\Big(p_2+\frac{n-2}{4(n-1)}r_2\Big)_{(\bar{g},\bar{\bar{g}})}f^{\frac{n-2}{4}}.
\end{align}
\end{thm}
\section{Appendix}
In this section, we list some complex coefficients used in Sction \ref{section:3} and Sction \ref{section:4}.

$\mathbf{(1)}$ The coefficients of $\varepsilon^0,\varepsilon,\varepsilon^2$ in the scalar curvature $r$:
\begin{align}\label{b1}
r_0&=\bar{g}^{jl}\partial_k(\bar{\Gamma}^k_{jl})+\bar{g}^{jl}\bar{\Gamma}^\alpha_{jl}\bar{\Gamma}^k_{k\alpha}-\bar{g}^{jl}\partial_j(\bar{\Gamma}^k_{kl})+\bar{g}^{jl}\bar{\Gamma}^\alpha_{kl}\bar{\Gamma}^k_{j\alpha};\nonumber\\
r_1&=-\partial_k(\bar{\Gamma}^k_{jl})\bar{g}^{j\lambda}\bar{\bar{g}}_{\lambda\sigma}\bar{g}^{\sigma l}+\frac{1}{2}\bar{g}^{jl}\partial_k(\bar{g}^{kr})\Big(\frac{\partial \bar{\bar{g}}_{jr}}{\partial{x_l}}+\frac{\partial \bar{\bar{g}}_{lr}}{\partial{x_j}}-\frac{\partial \bar{\bar{g}}_{jl}}{\partial{x_r}}\Big)+\frac{1}{2}\bar{g}^{jl}\bar{g}^{kr}\Big(\frac{\partial^2 \bar{\bar{g}}_{jr}}{\partial{x_k}\partial{x_l}}+\frac{\partial ^2 \bar{\bar{g}}_{lr}}{\partial{x_k}\partial{x_j}}-\frac{\partial^2 \bar{\bar{g}}_{jl}}{\partial{x_k}\partial{x_r}}\Big)\nonumber\\
&-\frac{1}{2}\bar{g}^{jl}\partial_k(\bar{g}^{k\lambda}\bar{\bar{g}}_{\lambda\sigma}\bar{g}^{\sigma r})\Big(\frac{\partial \bar{g}^{jr}}{\partial{x_l}}+\frac{\partial \bar{g}^{lr}}{\partial{x_j}}-\frac{\partial \bar{g}^{jl}}{\partial{x_r}}\Big)-\frac{1}{2}\bar{g}^{jl}\bar{g}^{k\lambda}\bar{\bar{g}}_{\lambda\sigma}\bar{g}^{\sigma r}\Big(\frac{\partial^2 \bar{g}^{jr}}{\partial{x_k}\partial{x_l}}+\frac{\partial ^2 \bar{g}^{lr}}{\partial{x_k}\partial{x_j}}-\frac{\partial^2 \bar{g}^{jl}}{\partial{x_k}\partial{x_r}}\Big)\nonumber\\
&-\bar{g}^{j\lambda}\bar{\bar{g}}_{\lambda\sigma}\bar{g}^{\sigma l}\bar{\Gamma}^\alpha_{jl}\bar{\Gamma}^k_{k\alpha}+\frac{1}{2}\bar{g}^{jl}\bar{\Gamma}^\alpha_{jl}\Big[\bar{g}^{kr}\Big(\frac{\partial \bar{\bar{g}}_{kr}}{\partial{x_\alpha}}+\frac{\partial \bar{\bar{g}}_{\alpha r}}{\partial{x_k}}-\frac{\partial \bar{\bar{g}}_{k\alpha}}{\partial{x_r}}\Big)-\bar{g}^{k\lambda}\bar{\bar{g}}_{\lambda\sigma}\bar{g}^{\sigma r}\Big(\frac{\partial \bar{g}^{kr}}{\partial{x_\alpha}}+\frac{\partial \bar{g}^{\alpha r}}{\partial{x_k}}-\frac{\partial \bar{g}^{k \alpha}}{\partial{x_r}}\Big)\Big]\nonumber\\
&+\frac{1}{2}\bar{g}^{jl}\Big[\bar{g}^{\alpha r}\Big(\frac{\partial \bar{\bar{g}}_{jr}}{\partial{x_l}}+\frac{\partial \bar{\bar{g}}_{l r}}{\partial{x_j}}-\frac{\partial \bar{\bar{g}}_{jl}}{\partial{x_r}}\Big)-\bar{g}^{\alpha\lambda}\bar{\bar{g}}_{\lambda\sigma}\bar{g}^{\sigma r}\Big(\frac{\partial \bar{g}^{jr}}{\partial{x_l}}+\frac{\partial \bar{g}^{l r}}{\partial{x_j}}-\frac{\partial \bar{g}^{jl}}{\partial{x_r}}\Big)\Big]\bar{\Gamma}^k_{k\alpha}+\bar{g}^{j\lambda}\bar{\bar{g}}_{\lambda\sigma}\bar{g}^{\sigma l}\partial_j(\bar{\Gamma}^k_{kl})\nonumber\\
&-\frac{1}{2}\bar{g}^{jl}\partial_k(\bar{g}^{kr})\Big(\frac{\partial \bar{\bar{g}}_{kr}}{\partial{x_l}}+\frac{\partial \bar{\bar{g}}_{lr}}{\partial{x_k}}-\frac{\partial \bar{\bar{g}}_{kl}}{\partial{x_r}}\Big)-\frac{1}{2}\bar{g}^{jl}\bar{g}^{kr}\Big(\frac{\partial^2 \bar{\bar{g}}_{kr}}{\partial{x_j}\partial{x_l}}+\frac{\partial ^2 \bar{\bar{g}}_{lr}}{\partial{x_j}\partial{x_k}}-\frac{\partial^2 \bar{\bar{g}}_{kl}}{\partial{x_j}\partial{x_r}}\Big)+\frac{1}{2}\bar{g}^{jl}\partial_j(\bar{g}^{k\lambda}\bar{\bar{g}}_{\lambda\sigma}\bar{g}^{\sigma r})\nonumber\\
&\Big(\frac{\partial \bar{g}^{kr}}{\partial{x_l}}+\frac{\partial \bar{g}^{lr}}{\partial{x_k}}-\frac{\partial \bar{g}^{kl}}{\partial{x_r}}\Big)+\frac{1}{2}\bar{g}^{jl}\bar{g}^{k\lambda}\bar{\bar{g}}_{\lambda\sigma}\bar{g}^{\sigma r}\Big(\frac{\partial ^2 \bar{g}^{kr}}{\partial{x_j}\partial{x_l}}+\frac{\partial^2 \bar{g}^{lr}}{\partial{x_j}\partial{x_k}}-\frac{\partial ^2 \bar{g}^{kl}}{\partial{x_j}\partial{x_r}}\Big)+\bar{g}^{j\lambda}\bar{\bar{g}}_{\lambda\sigma}\bar{g}^{\sigma l}\bar{\Gamma}^\alpha_{kl}\bar{\Gamma}^k_{j\alpha}\nonumber\\
&-\frac{1}{2}\bar{g}^{jl}\Big[\bar{g}^{\alpha r}\Big(\frac{\partial \bar{\bar{g}}_{kr}}{\partial{x_l}}+\frac{\partial \bar{\bar{g}}_{l r}}{\partial{x_k}}-\frac{\partial \bar{\bar{g}}_{kl}}{\partial{x_r}}\Big)-\bar{g}^{\alpha\lambda}\bar{\bar{g}}_{\lambda\sigma}\bar{g}^{\sigma r}\Big(\frac{\partial \bar{g}^{kr}}{\partial{x_l}}+\frac{\partial \bar{g}^{l r}}{\partial{x_k}}-\frac{\partial \bar{g}^{kl}}{\partial{x_r}}\Big)\Big]\bar{\Gamma}^k_{j\alpha}\nonumber\\
&-\frac{1}{2}\bar{g}^{jl}\bar{\Gamma}^\alpha_{kl}\Big[\bar{g}^{k r}\Big(\frac{\partial \bar{\bar{g}}_{jr}}{\partial{x_\alpha}}+\frac{\partial \bar{\bar{g}}_{\alpha r}}{\partial{x_j}}-\frac{\partial \bar{\bar{g}}_{j\alpha}}{\partial{x_r}}\Big)-\bar{g}^{k\lambda}\bar{\bar{g}}_{\lambda\sigma}\bar{g}^{\sigma r}\Big(\frac{\partial \bar{g}^{jr}}{\partial{x_\alpha}}+\frac{\partial \bar{g}^{\alpha r}}{\partial{x_j}}-\frac{\partial \bar{g}^{j\alpha}}{\partial{x_r}}\Big)\Big];\nonumber\\
r_2&=\partial_k(\bar{\Gamma}^k_{jl})\bar{g}^{j\lambda}\bar{\bar{g}}_{\lambda\sigma}\bar{g}^{\sigma \alpha}\bar{\bar{g}}_{\alpha\beta}\bar{g}^{\beta l}-\frac{1}{2}\bar{g}^{jl}\partial_k(\bar{g}^{k\lambda}\bar{\bar{g}}_{\lambda \sigma}\bar{g}^{\sigma r})\Big(\frac{\partial \bar{\bar{g}}_{jr}}{\partial{x_l}}+\frac{\partial \bar{\bar{g}}_{lr}}{\partial{x_j}}-\frac{\partial \bar{\bar{g}}_{jl}}{\partial{x_r}}\Big)-\frac{1}{2}\bar{g}^{jl}\bar{g}^{k\lambda}\bar{\bar{g}}_{\lambda\sigma}\bar{g}^{\sigma r}\Big(\frac{\partial^2 \bar{\bar{g}}_{jr}}{\partial{x_k}\partial{x_l}}\nonumber\\
&+\frac{\partial ^2 \bar{\bar{g}}_{lr}}{\partial{x_k}\partial{x_j}}-\frac{\partial^2 \bar{\bar{g}}_{jl}}{\partial{x_k}\partial{x_r}}\Big)+\frac{1}{2}\bar{g}^{jl}\Big[\partial_k(\bar{g}^{k\lambda}\bar{\bar{g}}_{\lambda \sigma}\bar{g}^{\sigma \alpha}\bar{\bar{g}}_{\alpha\beta}\bar{g}^{\beta r})\Big(\frac{\partial \bar{g}^{jr}}{\partial{x_l}}+\frac{\partial \bar{g}^{lr}}{\partial{x_j}}-\frac{\partial \bar{g}^{jl}}{\partial{x_r}}\Big)+\bar{g}^{k\lambda}\bar{\bar{g}}_{\lambda \sigma}\bar{g}^{\sigma \alpha}\bar{\bar{g}}_{\alpha\beta}\bar{g}^{\beta r}\nonumber\\
&\Big(\frac{\partial^2 \bar{\bar{g}}_{jr}}{\partial{x_k}\partial{x_l}}+\frac{\partial ^2 \bar{\bar{g}}_{lr}}{\partial{x_k}\partial{x_j}}-\frac{\partial^2 \bar{\bar{g}}_{jl}}{\partial{x_k}\partial{x_r}}\Big)\Big]-\frac{1}{2}\bar{g}^{j\lambda}\bar{\bar{g}}_{\lambda \sigma}\bar{g}^{\sigma l}\Big[\partial_k(\bar{g}^{kr})\Big(\frac{\partial \bar{\bar{g}}_{jr}}{\partial{x_l}}+\frac{\partial \bar{\bar{g}}_{lr}}{\partial{x_j}}-\frac{\partial \bar{\bar{g}}_{jl}}{\partial{x_r}}\Big)+\bar{g}^{kr}\Big(\frac{\partial^2 \bar{\bar{g}}_{jr}}{\partial{x_k}\partial{x_l}}\nonumber\\
&+\frac{\partial ^2 \bar{\bar{g}}_{lr}}{\partial{x_k}\partial{x_j}}-\frac{\partial^2 \bar{\bar{g}}_{jl}}{\partial{x_k}\partial{x_r}}\Big)+\bar{g}_{jl}\partial_k(\bar{g}^{k\lambda}\bar{\bar{g}}_{\lambda \sigma}\bar{g}^{\sigma r})\Big(\frac{\partial \bar{g}^{jr}}{\partial{x_l}}+\frac{\partial \bar{g}^{lr}}{\partial{x_j}}-\frac{\partial \bar{g}^{jl}}{\partial{x_r}}\Big)-\bar{g}_{jl}\bar{g}^{k\lambda}\bar{\bar{g}}_{\lambda \sigma}\bar{g}^{\sigma r}\Big(\frac{\partial^2 \bar{\bar{g}}_{jr}}{\partial{x_k}\partial{x_l}}\nonumber\\
&+\frac{\partial ^2 \bar{\bar{g}}_{lr}}{\partial{x_k}\partial{x_j}}-\frac{\partial^2 \bar{\bar{g}}_{jl}}{\partial{x_k}\partial{x_r}}\Big)\Big]+\bar{g}^{j\lambda}\bar{\bar{g}}_{\lambda \sigma}\bar{g}^{\sigma \alpha}\bar{\bar{g}}_{\alpha\beta}\bar{g}^{\beta l}\bar{\Gamma}^\alpha_{jl}\bar{\Gamma}^k_{k\alpha}+\frac{1}{2}\bar{g}_{jl}\Big[-\bar{g}^{\alpha\lambda}\bar{\bar{g}}_{\lambda \sigma}\bar{g}^{\sigma r}\Big(\frac{\partial \bar{\bar{g}}_{jr}}{\partial{x_l}}+\frac{\partial \bar{\bar{g}}_{lr}}{\partial{x_j}}-\frac{\partial \bar{\bar{g}}_{jl}}{\partial{x_r}}\Big)\nonumber\\
&+\bar{g}^{\alpha\lambda}\bar{\bar{g}}_{\lambda \sigma}\bar{g}^{\sigma k}\bar{\bar{g}}_{k\beta}\bar{g}^{\beta r}\Big(\frac{\partial \bar{g}^{jr}}{\partial{x_l}}+\frac{\partial \bar{g}^{lr}}{\partial{x_j}}-\frac{\partial \bar{g}^{jl}}{\partial{x_r}}\Big)\Big]\bar{\Gamma}^k_{k\alpha}+\frac{1}{2}\bar{g}_{jl}\bar{\Gamma}^\alpha_{jl}\Big[-\bar{g}^{k\lambda}\bar{\bar{g}}_{\lambda \sigma}\bar{g}^{\sigma r}\Big(\frac{\partial \bar{\bar{g}}_{kr}}{\partial{x_l}}+\frac{\partial \bar{\bar{g}}_{lr}}{\partial{x_k}}-\frac{\partial \bar{\bar{g}}_{kl}}{\partial{x_r}}\Big)\nonumber\\
&+\bar{g}^{k\lambda}\bar{\bar{g}}_{\lambda \sigma}\bar{g}^{\sigma m}\bar{\bar{g}}_{m\beta}\bar{g}^{\beta r}\Big(\frac{\partial \bar{g}^{kr}}{\partial{x_l}}+\frac{\partial \bar{g}^{lr}}{\partial{x_k}}-\frac{\partial \bar{g}^{kl}}{\partial{x_r}}\Big)\Big]+\frac{1}{4}\bar{g}_{jl}\Big[\bar{g}^{k r}\Big(\frac{\partial \bar{\bar{g}}_{kr}}{\partial{x_l}}+\frac{\partial \bar{\bar{g}}_{lr}}{\partial{x_k}}-\frac{\partial \bar{\bar{g}}_{kl}}{\partial{x_r}}\Big)-\bar{g}^{k\lambda}\bar{\bar{g}}_{\lambda \sigma}\bar{g}^{\sigma r}\nonumber\\
&\Big(\frac{\partial \bar{g}^{kr}}{\partial{x_l}}+\frac{\partial \bar{g}^{lr}}{\partial{x_k}}-\frac{\partial \bar{g}^{kl}}{\partial{x_r}}\Big)\Big]\Big[\bar{g}^{\alpha r}\Big(\frac{\partial \bar{\bar{g}}_{jr}}{\partial{x_l}}+\frac{\partial \bar{\bar{g}}_{lr}}{\partial{x_j}}-\frac{\partial \bar{\bar{g}}_{jl}}{\partial{x_r}}\Big)-\bar{g}^{\alpha\lambda}\bar{\bar{g}}_{\lambda \sigma}\bar{g}^{\sigma r}\Big(\frac{\partial \bar{g}^{jr}}{\partial{x_l}}+\frac{\partial \bar{g}^{lr}}{\partial{x_j}}-\frac{\partial \bar{g}^{jl}}{\partial{x_r}}\Big)\Big]-\frac{1}{2}\bar{g}_{j\lambda}\bar{\bar{g}}_{\lambda \sigma}\nonumber\\
&\bar{g}^{\sigma l}\bar{\Gamma}^\alpha_{jl}\Big[\bar{g}^{\alpha r}\Big(\frac{\partial \bar{\bar{g}}_{jr}}{\partial{x_l}}+\frac{\partial \bar{\bar{g}}_{lr}}{\partial{x_j}}-\frac{\partial \bar{\bar{g}}_{jl}}{\partial{x_r}}\Big)-\bar{g}^{\alpha\lambda}\bar{\bar{g}}_{\lambda \sigma}\bar{g}^{\sigma r}\Big(\frac{\partial \bar{g}^{jr}}{\partial{x_l}}+\frac{\partial \bar{g}^{lr}}{\partial{x_j}}-\frac{\partial \bar{g}^{jl}}{\partial{x_r}}\Big)\Big]-\frac{1}{2}\bar{g}_{j\lambda}\bar{\bar{g}}_{\lambda \sigma}\bar{g}^{\sigma l}\Big[\bar{g}^{k r}\Big(\frac{\partial \bar{\bar{g}}_{kr}}{\partial{x_l}}+\frac{\partial \bar{\bar{g}}_{lr}}{\partial{x_k}}\nonumber\\
&-\frac{\partial \bar{\bar{g}}_{kl}}{\partial{x_r}}\Big)-\bar{g}^{k\lambda}\bar{\bar{g}}_{\lambda \sigma}\bar{g}^{\sigma r}\Big(\frac{\partial \bar{g}^{kr}}{\partial{x_l}}
+\frac{\partial \bar{g}^{lr}}{\partial{x_k}}-\frac{\partial \bar{g}^{kl}}{\partial{x_r}}\Big)\Big]\bar{\Gamma}^k_{k\alpha}+\frac{1}{2}\bar{g}^{jl}\partial_k(\bar{g}^{k\lambda}\bar{\bar{g}}_{\lambda \sigma}\bar{g}^{\sigma r})\Big(\frac{\partial \bar{\bar{g}}_{kr}}{\partial{x_l}}+\frac{\partial \bar{\bar{g}}_{lr}}{\partial{x_k}}-\frac{\partial \bar{\bar{g}}_{kl}}{\partial{x_r}}\Big)+\frac{1}{2}\bar{g}^{jl}\nonumber\\
&\bar{g}^{k\lambda}\bar{\bar{g}}_{\lambda\sigma}\bar{g}^{\sigma r}\Big(\frac{\partial^2 \bar{\bar{g}}_{kr}}{\partial{x_j}\partial{x_l}}+\frac{\partial ^2 \bar{\bar{g}}_{lr}}{\partial{x_j}\partial{x_k}}-\frac{\partial^2 \bar{\bar{g}}_{kl}}{\partial{x_j}\partial{x_r}}\Big)-\frac{1}{2}\bar{g}^{jl}\partial_k(\bar{g}^{k\lambda}\bar{\bar{g}}_{\lambda \sigma}\bar{g}^{\sigma \alpha}\bar{\bar{g}}_{\alpha\beta}\bar{g}^{\beta r})\Big(\frac{\partial \bar{g}^{kr}}{\partial{x_l}}+\frac{\partial \bar{g}^{lr}}{\partial{x_k}}-\frac{\partial \bar{g}^{kl}}{\partial{x_r}}\Big)-\frac{1}{2}\bar{g}^{jl}\nonumber\\
&\bar{g}^{k\lambda}\bar{\bar{g}}_{\lambda \sigma}\bar{g}^{\sigma \alpha}\bar{\bar{g}}_{\alpha\beta}\bar{g}^{\beta r}\Big(\frac{\partial^2 \bar{g}^{kr}}{\partial{x_j}\partial{x_l}}+\frac{\partial ^2 \bar{g}^{lr}}{\partial{x_j}\partial{x_k}}-\frac{\partial^2 \bar{g}^{kl}}{\partial{x_j}\partial{x_r}}\Big)-\bar{g}^{j\lambda}\bar{\bar{g}}_{\lambda \sigma}\bar{g}^{\sigma \alpha}\bar{\bar{g}}_{\alpha\beta}\bar{g}^{\beta l}\partial_j(\bar{\Gamma}^k_{kl})+\frac{1}{2}\bar{g}_{j\lambda}\bar{\bar{g}}_{\lambda \sigma}\bar{g}^{\sigma l}\Big[\partial_j(\bar{g}^{k r})\nonumber\\
&\Big(\frac{\partial \bar{\bar{g}}_{kr}}{\partial{x_l}}+\frac{\partial \bar{\bar{g}}_{lr}}{\partial{x_k}}-\frac{\partial \bar{\bar{g}}_{kl}}{\partial{x_r}}\Big)+\bar{g}^{k r}\frac{\partial^2 \bar{\bar{g}}_{kr}}{\partial{x_j}\partial{x_l}}+\frac{\partial ^2 \bar{\bar{g}}_{lr}}{\partial{x_j}\partial{x_k}}-\frac{\partial^2 \bar{\bar{g}}_{kl}}{\partial{x_j}\partial{x_r}}\Big)-\partial_j(\bar{g}^{k\lambda}\bar{\bar{g}}_{\lambda \sigma}\bar{g}^{\sigma r})\Big(\frac{\partial \bar{g}^{kr}}{\partial{x_l}}+\frac{\partial \bar{g}^{lr}}{\partial{x_k}}-\frac{\partial \bar{g}^{kl}}{\partial{x_r}}\Big)\nonumber\\
&-\bar{g}^{k\lambda}\bar{\bar{g}}_{\lambda \sigma}\bar{g}^{\sigma r}\Big(\frac{\partial^2 \bar{g}^{kr}}{\partial{x_j}\partial{x_l}}+\frac{\partial ^2 \bar{g}^{lr}}{\partial{x_j}\partial{x_k}}-\frac{\partial^2 \bar{g}^{kl}}{\partial{x_j}\partial{x_r}}\Big)\Big]-\bar{g}^{j\lambda}\bar{\bar{g}}_{\lambda \sigma}\bar{g}^{\sigma \alpha}\bar{\bar{g}}_{\alpha\beta}\bar{g}^{\beta l}\bar{\Gamma}^\alpha_{kl}\bar{\Gamma}^k_{j\alpha}-\frac{1}{2}\bar{g}_{jl}\bar{\Gamma}^\alpha_{kl}\Big[-\bar{g}^{k\lambda}\bar{\bar{g}}_{\lambda \sigma}\bar{g}^{\sigma r}\nonumber\\
&\Big(\frac{\partial \bar{\bar{g}}_{jr}}{\partial{x_l}}+\frac{\partial \bar{\bar{g}}_{lr}}{\partial{x_j}}-\frac{\partial \bar{\bar{g}}_{jl}}{\partial{x_r}}\Big)+\bar{g}^{\alpha\lambda}\bar{\bar{g}}_{\lambda \sigma}\bar{g}^{\sigma m}\bar{\bar{g}}_{m\beta}\bar{g}^{\beta r}\Big(\frac{\partial \bar{g}^{kr}}{\partial{x_l}}+\frac{\partial \bar{g}^{lr}}{\partial{x_k}}-\frac{\partial \bar{g}^{kl}}{\partial{x_r}}\Big)\Big]-\frac{1}{2}\bar{g}_{jl}\Big[-\bar{g}^{\alpha\lambda}\bar{\bar{g}}_{\lambda \sigma}\bar{g}^{\sigma r}\Big(\frac{\partial \bar{\bar{g}}_{kr}}{\partial{x_l}}+\frac{\partial \bar{\bar{g}}_{lr}}{\partial{x_k}}\nonumber\\
&-\frac{\partial \bar{\bar{g}}_{kl}}{\partial{x_r}}\Big)+\bar{g}^{\alpha\lambda}\bar{\bar{g}}_{\lambda \sigma}\bar{g}^{\sigma m}\bar{\bar{g}}_{m\beta}\bar{g}^{\beta r}\Big(\frac{\partial \bar{g}^{kr}}{\partial{x_l}}+\frac{\partial \bar{g}^{lr}}{\partial{x_k}}-\frac{\partial \bar{g}^{kl}}{\partial{x_r}}\Big)\Big]\bar{\Gamma}^k_{kl}+\frac{1}{4}\bar{g}_{jl}\Big[\bar{g}^{\alpha r}\Big(\frac{\partial \bar{\bar{g}}_{kr}}{\partial{x_l}}+\frac{\partial \bar{\bar{g}}_{lr}}{\partial{x_k}}-\frac{\partial \bar{\bar{g}}_{kl}}{\partial{x_r}}\Big)-\bar{g}^{\alpha\lambda}\bar{\bar{g}}_{\lambda \sigma}\nonumber\\
&\bar{g}^{\sigma r}\Big(\frac{\partial \bar{g}^{jr}}{\partial{x_l}}+\frac{\partial \bar{g}^{lr}}{\partial{x_j}}-\frac{\partial \bar{g}^{jl}}{\partial{x_r}}\Big)\Big]\Big[\bar{g}^{k r}\Big(\frac{\partial \bar{\bar{g}}_{jr}}{\partial{x_l}}+\frac{\partial \bar{\bar{g}}_{lr}}{\partial{x_j}}-\frac{\partial \bar{\bar{g}}_{jl}}{\partial{x_r}}\Big)-\bar{g}^{k\lambda}\bar{\bar{g}}_{\lambda \sigma}\bar{g}^{\sigma r}\Big(\frac{\partial \bar{g}^{jr}}{\partial{x_l}}+\frac{\partial \bar{g}^{lr}}{\partial{x_j}}-\frac{\partial \bar{g}^{jl}}{\partial{x_r}}\Big)\Big]+\frac{1}{2}\bar{g}_{j\lambda}\bar{\bar{g}}_{\lambda \sigma}\nonumber\\
&\bar{g}^{\sigma l}\bar{\Gamma}^\alpha_{kl}\Big[\bar{g}^{k r}\Big(\frac{\partial \bar{\bar{g}}_{jr}}{\partial{x_l}}+\frac{\partial \bar{\bar{g}}_{lr}}{\partial{x_j}}-\frac{\partial \bar{\bar{g}}_{jl}}{\partial{x_r}}\Big)-\bar{g}^{k\lambda}\bar{\bar{g}}_{\lambda \sigma}\bar{g}^{\sigma r}\Big(\frac{\partial \bar{g}^{jr}}{\partial{x_l}}+\frac{\partial \bar{g}^{lr}}{\partial{x_j}}-\frac{\partial \bar{g}^{jl}}{\partial{x_r}}\Big)\Big]+\frac{1}{2}\bar{g}_{j\lambda}\bar{\bar{g}}_{\lambda \sigma}\bar{g}^{\sigma l}\Big[\bar{g}^{\alpha r}\Big(\frac{\partial \bar{\bar{g}}_{kr}}{\partial{x_l}}+\frac{\partial \bar{\bar{g}}_{lr}}{\partial{x_k}}\nonumber\\
&-\frac{\partial \bar{\bar{g}}_{kl}}{\partial{x_r}}\Big)-\bar{g}^{\alpha\lambda}\bar{\bar{g}}_{\lambda \sigma}\bar{g}^{\sigma r}\Big(\frac{\partial \bar{g}^{kr}}{\partial{x_l}}+\frac{\partial \bar{g}^{lr}}{\partial{x_k}}-\frac{\partial \bar{g}^{kl}}{\partial{x_r}}\Big)\Big]\bar{\Gamma}^k_{j\alpha}.
\end{align}

$\mathbf{(2)}$The coefficients of $\varepsilon^0,\varepsilon,\varepsilon^2$ in $\Delta<df_1,df_2>$:
\begin{align}\label{b2}
a_0&=\Big[-\sum_{\alpha\beta}\bar{g}^{\alpha\beta}(x)(\partial_\alpha\partial_\beta-\sum_k\bar{\Gamma}^k_{\alpha\beta}\partial_k)\Big]\Big(\sum_{jl}\frac{\partial f_1}{\partial x_j}\frac{\partial f_2}{\partial x_l}\bar{g}^{jl}\Big)
;\nonumber\\
a_1&=\sum_{\alpha\beta}\bar{g}^{\alpha\beta}(x)\partial_\alpha\partial_\beta\Big(\frac{\partial f_1}{\partial x_j}\frac{\partial f_2}{\partial x_l}\bar{g}^{j\lambda}\bar{\bar{g}}_{\lambda\sigma}\bar{g}^{\sigma l}\Big)-\sum_{\alpha\beta}\bar{g}^{\alpha\beta}(x)\sum_k\bar{\Gamma}^k_{\alpha\beta}\partial_k\Big(\frac{\partial f_1}{\partial x_j}\frac{\partial f_2}{\partial x_l}\bar{g}^{j\lambda}\bar{\bar{g}}_{\lambda\sigma}\bar{g}^{\sigma l}\Big)\nonumber\\
&+\frac{1}{2}\sum_{\alpha\beta}\bar{g}^{\alpha\beta}(x)\sum_k\Big[\bar{g}^{kl}\Big(\frac{\partial \bar{\bar{g}}_{\alpha l}}{\partial{x_\beta}}+\frac{\partial \bar{\bar{g}}_{\beta l}}{\partial{x_\alpha}}-\frac{\partial \bar{\bar{g}}_{\alpha\beta}}{\partial{x_l}}\Big)-\bar{g}^{k\lambda}\bar{\bar{g}}_{\lambda\sigma}\bar{g}^{\sigma l}\Big(\frac{\partial \bar{g}^{\alpha l}}{\partial{x_\beta}}+\frac{\partial \bar{g}^{\beta l}}{\partial{x_\alpha}}-\frac{\partial \bar{g}^{\alpha\beta}}{\partial{x_l}}\Big)\Big]\partial_k\Big(\frac{\partial f_1}{\partial x_j}\frac{\partial f_2}{\partial x_l}\bar{g}^{jl}\Big)\nonumber\\
&+\sum_{\alpha\beta}\bar{g}^{\alpha\lambda}\bar{\bar{g}}_{\lambda\sigma}\bar{g}^{\sigma \beta}\partial_\alpha\partial_\beta\Big(\frac{\partial f_1}{\partial x_j}\frac{\partial f_2}{\partial x_l}\bar{g}^{jl}\Big)-\sum_{\alpha\beta}\bar{g}^{\alpha\lambda}\bar{\bar{g}}_{\lambda\sigma}\bar{g}^{\sigma \beta}\sum_k\bar{\Gamma}^k_{\alpha\beta}\partial_k\Big(\frac{\partial f_1}{\partial x_j}\frac{\partial f_2}{\partial x_l}\bar{g}^{jl}\Big);\nonumber\\
a_2&=\sum_{\alpha\beta}\bar{g}^{\alpha\lambda}\bar{\bar{g}}_{\lambda\sigma}\bar{g}^{\sigma l}\bar{\bar{g}}_{lj}\bar{g}^{j\beta}\partial_\alpha\partial_\beta\Big(\frac{\partial f_1}{\partial x_j}\frac{\partial f_2}{\partial x_l}\bar{g}^{jl}\Big)+\sum_{\alpha\beta}\bar{g}^{\alpha\beta}(x)\partial_\alpha\partial_\beta\Big(\frac{\partial f_1}{\partial x_j}\frac{\partial f_2}{\partial x_l}\bar{g}^{j\lambda}\bar{\bar{g}}_{\lambda\sigma}\bar{g}^{\sigma \alpha}\bar{\bar{g}}_{\alpha\beta}\bar{g}^{\beta l}\Big)-\sum_{\alpha\beta}\bar{g}^{\alpha\lambda}\bar{\bar{g}}_{\lambda\sigma}\bar{g}^{\sigma l}\nonumber\\
&\bar{\bar{g}}_{lj}\bar{g}^{j\beta}\sum_k\bar{\Gamma}^k_{\alpha\beta}\partial_k\Big(\frac{\partial f_1}{\partial x_j}\frac{\partial f_2}{\partial x_l}\bar{g}^{jl}\Big)-\sum_{\alpha\beta}\bar{g}^{\alpha\beta}(x)\sum_k\bar{\Gamma}^k_{\alpha\beta}\partial_k\Big(\frac{\partial f_1}{\partial x_j}\frac{\partial f_2}{\partial x_l}\bar{g}^{j\lambda}\bar{\bar{g}}_{\lambda\sigma}\bar{g}^{\sigma \alpha}\bar{\bar{g}}_{\alpha\beta}\bar{g}^{\beta l}\Big)+\frac{1}{2}\bar{g}^{\alpha\beta}(x)\sum_k\Big[-\bar{g}^{k\lambda}\bar{\bar{g}}_{\lambda\sigma}\nonumber\\
&\bar{g}^{\sigma l}\Big(\frac{\partial \bar{\bar{g}}_{\alpha l}}{\partial{x_\beta}}+\frac{\partial \bar{\bar{g}}_{\beta l}}{\partial{x_\alpha}}-\frac{\partial \bar{\bar{g}}_{\alpha\beta}}{\partial{x_l}}\Big)+\bar{g}^{k\lambda}\bar{\bar{g}}_{\lambda\sigma}\bar{g}^{\sigma i}\bar{\bar{g}}_{ij}\bar{g}^{j l}\Big(\frac{\partial \bar{g}^{\alpha l}}{\partial{x_\beta}}+\frac{\partial\bar{g}^{\beta l}}{\partial{x_\alpha}}-\frac{\partial \bar{g}^{\alpha\beta}}{\partial{x_l}}\Big)\Big]\partial_k\Big(\frac{\partial f_1}{\partial x_j}\frac{\partial f_2}{\partial x_l}\bar{g}^{jl}\Big)-\sum_{\alpha\beta}\bar{g}^{\alpha\lambda}\bar{\bar{g}}_{\lambda\sigma}\bar{g}^{\sigma \beta}\nonumber\\
&\partial_\alpha\partial_\beta\Big(\frac{\partial f_1}{\partial x_j}\frac{\partial f_2}{\partial x_l}\bar{g}^{j\lambda}\bar{\bar{g}}_{\lambda\sigma}\bar{g}^{\sigma l}\Big)-\sum_{\alpha\beta}\bar{g}^{\alpha\lambda}\bar{\bar{g}}_{\lambda\sigma}\bar{g}^{\sigma \beta}\sum_k\bar{\Gamma}^k_{\alpha\beta}\partial_k\Big(\frac{\partial f_1}{\partial x_j}\frac{\partial f_2}{\partial x_l}\bar{g}^{j\lambda}\bar{\bar{g}}_{\lambda\sigma}\bar{g}^{\sigma l}\Big)+\frac{1}{2}\sum_{\alpha\beta}\bar{g}^{\alpha\lambda}\bar{\bar{g}}_{\lambda\sigma}\bar{g}^{\sigma \beta}\sum_k\Big[\bar{g}^{kl}\nonumber\\
&\Big(\frac{\partial \bar{\bar{g}}_{\alpha l}}{\partial{x_\beta}}+\frac{\partial \bar{\bar{g}}_{\beta l}}{\partial{x_\alpha}}-\frac{\partial \bar{\bar{g}}_{\alpha\beta}}{\partial{x_l}}\Big)-\bar{g}^{k\lambda}\bar{\bar{g}}_{\lambda\sigma}\bar{g}^{\sigma l}\Big(\frac{\partial \bar{g}^{\alpha l}}{\partial{x_\beta}}+\frac{\partial\bar{g}^{\beta l}}{\partial{x_\alpha}}-\frac{\partial \bar{g}^{\alpha\beta}}{\partial{x_l}}\Big)\Big]\partial_k\Big(\frac{\partial f_1}{\partial x_j}\frac{\partial f_2}{\partial x_l}\bar{g}^{jl}\Big)+\frac{1}{2}\sum_{\alpha\beta}\bar{g}^{\alpha\beta}(x)\nonumber\\
&\sum_k\Big[\bar{g}^{kl}\Big(\frac{\partial \bar{\bar{g}}_{\alpha l}}{\partial{x_\beta}}+\frac{\partial \bar{\bar{g}}_{\beta l}}{\partial{x_\alpha}}-\frac{\partial \bar{\bar{g}}_{\alpha\beta}}{\partial{x_l}}\Big)-\bar{g}^{k\lambda}\bar{\bar{g}}_{\lambda\sigma}\bar{g}^{\sigma l}\Big(\frac{\partial \bar{g}^{\alpha l}}{\partial{x_\beta}}+\frac{\partial\bar{g}^{\beta l}}{\partial{x_\alpha}}-\frac{\partial \bar{g}^{\alpha\beta}}{\partial{x_l}}\Big)\Big]\partial_k\Big(\frac{\partial f_1}{\partial x_j}\frac{\partial f_2}{\partial x_l}\bar{g}^{\alpha\lambda}\bar{\bar{g}}_{\lambda\sigma}\bar{g}^{\sigma l}\Big).
\end{align}

$\mathbf{(3)}$The coefficients of $\varepsilon^0,\varepsilon,\varepsilon^2$ in $<\nabla f_1,\nabla f_2>$:
\begin{align}\label{b3}
b_0&=\sum_{l\beta pq}\Big(\frac{\partial^2f_1}{\partial x_l\partial x_\beta}-\sum_k\frac{\partial f_1}{\partial x_k}\bar{\Gamma}^k_{\beta l}\Big)\Big(\frac{\partial^2f_2}{\partial x_p\partial x_q}-\sum_m\frac{\partial f_2}{\partial x_m}\bar{\Gamma}^m_{pq}\Big)\bar{g}^{\beta p}\bar{g}^{lq};\nonumber\\
b_1&=-\sum_{l\beta pq}\frac{\partial^2f_1}{\partial x_l\partial x_\beta}\frac{\partial^2f_2}{\partial x_p\partial x_q}\Big(\bar{g}^{\beta\lambda}\bar{\bar{g}}_{\lambda\sigma}\bar{g}^{\sigma p}\bar{g}^{lq}+\bar{g}^{\beta p}\bar{g}^{l\lambda}\bar{\bar{g}}_{\lambda\sigma}\bar{g}^{\sigma q}\Big)+\sum_{l\beta pq}\frac{\partial^2f_1}{\partial x_l\partial x_\beta}\sum_m\frac{\partial f_2}{\partial x_m}\bar{\Gamma}^m_{pq}\Big(\bar{g}^{\beta\lambda}\bar{\bar{g}}_{\lambda\sigma}\bar{g}^{\sigma p}\bar{g}^{lq}\nonumber\\
&+\bar{g}^{\beta p}\bar{g}^{l\lambda}\bar{\bar{g}}_{\lambda\sigma}\bar{g}^{\sigma q}\Big)+\sum_{l\beta pq}\sum_k\frac{\partial f_1}{\partial x_k}\bar{\Gamma}^k_{\beta l}\frac{\partial^2f_2}{\partial x_p\partial x_q}\Big(\bar{g}^{\beta\lambda}\bar{\bar{g}}_{\lambda\sigma}\bar{g}^{\sigma p}\bar{g}^{lq}+\bar{g}^{\beta p}\bar{g}^{l\lambda}\bar{\bar{g}}_{\lambda\sigma}\bar{g}^{\sigma q}\Big)+\sum_{l\beta pq}\sum_k\frac{\partial f_1}{\partial x_k}\bar{\Gamma}^k_{\beta l}\sum_m\frac{\partial f_2}{\partial x_m}\bar{\Gamma}^m_{pq}\nonumber\\
&\Big(\bar{g}^{\beta\lambda}\bar{\bar{g}}_{\lambda\sigma}\bar{g}^{\sigma p}\bar{g}^{lq}+\bar{g}^{\beta p}\bar{g}^{l\lambda}\bar{\bar{g}}_{\lambda\sigma}\bar{g}^{\sigma q}\Big)-\frac{1}{2}\sum_{l\beta pq}\frac{\partial^2f_1}{\partial x_l\partial x_\beta}\sum_m\frac{\partial f_2}{\partial x_m}\Big[\bar{g}^{ml}\Big(\frac{\partial \bar{\bar{g}}_{p l}}{\partial{x_q}}+\frac{\partial \bar{\bar{g}}_{q l}}{\partial{x_p}}-\frac{\partial \bar{\bar{g}}_{pq}}{\partial{x_l}}\Big)-\bar{g}^{m\lambda}\bar{\bar{g}}_{\lambda\sigma}\bar{g}^{\sigma l}\nonumber\\
&\Big(\frac{\partial \bar{g}^{p l}}{\partial{x_q}}+\frac{\partial\bar{g}^{q l}}{\partial{x_p}}-\frac{\partial \bar{g}^{pq}}{\partial{x_l}}\Big)\Big]\bar{g}^{\beta p}\bar{g}^{lq}-\frac{1}{2}\sum_{l\beta pq}\sum_k\frac{\partial f_1}{\partial x_k}\Big[\bar{g}^{kj}\Big(\frac{\partial \bar{\bar{g}}_{\beta j}}{\partial{x_l}}+\frac{\partial \bar{\bar{g}}_{lj}}{\partial{x_\beta}}-\frac{\partial \bar{\bar{g}}_{\beta l}}{\partial{x_j}}\Big)-\bar{g}^{k\lambda}\bar{\bar{g}}_{\lambda\sigma}\bar{g}^{\sigma j}\Big(\frac{\partial \bar{g}^{\beta j}}{\partial{x_l}}+\frac{\partial\bar{g}^{lj}}{\partial{x_\beta}}\nonumber\\
&-\frac{\partial \bar{g}^{\beta l}}{\partial{x_j}}\Big)\Big]\frac{\partial^2f_2}{\partial x_p\partial x_q}\bar{g}^{\beta p}\bar{g}^{lq}+\frac{1}{2}\sum_{l\beta pq}\sum_k\frac{\partial f_1}{\partial x_k}\bar{\Gamma}^k_{\beta l}\sum_m\frac{\partial f_2}{\partial x_m}\Big[\bar{g}^{ml}\Big(\frac{\partial \bar{\bar{g}}_{p l}}{\partial{x_q}}+\frac{\partial \bar{\bar{g}}_{q l}}{\partial{x_p}}-\frac{\partial \bar{\bar{g}}_{pq}}{\partial{x_l}}\Big)-\bar{g}^{m\lambda}\bar{\bar{g}}_{\lambda\sigma}\bar{g}^{\sigma l}\Big(\frac{\partial \bar{g}^{p l}}{\partial{x_q}}\nonumber\\
&+\frac{\partial\bar{g}^{q l}}{\partial{x_p}}-\frac{\partial \bar{g}^{pq}}{\partial{x_l}}\Big)\Big]\bar{g}^{\beta p}\bar{g}^{lq}+\frac{1}{2}\sum_{l\beta pq}\sum_k\frac{\partial f_1}{\partial x_k}\Big[\bar{g}^{kj}\Big(\frac{\partial \bar{\bar{g}}_{\beta j}}{\partial{x_l}}+\frac{\partial \bar{\bar{g}}_{lj}}{\partial{x_\beta}}-\frac{\partial \bar{\bar{g}}_{\beta l}}{\partial{x_j}}\Big)-\bar{g}^{k\lambda}\bar{\bar{g}}_{\lambda\sigma}\bar{g}^{\sigma j}\Big(\frac{\partial \bar{g}^{\beta j}}{\partial{x_l}}+\frac{\partial\bar{g}^{lj}}{\partial{x_\beta}}-\frac{\partial \bar{g}^{\beta l}}{\partial{x_j}}\Big)\Big]\nonumber\\
&\frac{\partial^2f_2}{\partial x_p\partial x_q}\sum_m\frac{\partial f_2}{\partial x_m}\bar{\Gamma}^m_{pq}\bar{g}^{\beta p}\bar{g}^{lq};\nonumber\\
b_2&=\sum_{l\beta pq}\Big(\frac{\partial^2f_1}{\partial x_l\partial x_\beta}-\sum_k\frac{\partial f_1}{\partial x_k}\bar{\Gamma}^k_{\beta l}\Big)\Big(\frac{\partial^2f_2}{\partial x_p\partial x_q}-\sum_m\frac{\partial f_2}{\partial x_m}\bar{\Gamma}^m_{pq}\Big)\bar{g}^{\beta\lambda}\bar{\bar{g}}_{\lambda\sigma}\bar{g}^{\sigma \alpha}\bar{\bar{g}}_{\alpha\gamma}\bar{g}^{\gamma p}\bar{g}^{lq}+\sum_{l\beta pq}\Big(\frac{\partial^2f_1}{\partial x_l\partial x_\beta}-\sum_k\frac{\partial f_1}{\partial x_k}\nonumber\\
&\bar{\Gamma}^k_{\beta l}\Big)\Big(\frac{\partial^2f_2}{\partial x_p\partial x_q}-\sum_m\frac{\partial f_2}{\partial x_m}\bar{\Gamma}^m_{pq}\Big)\bar{g}^{\beta p}\bar{g}^{l\lambda}\bar{\bar{g}}_{\lambda\sigma}\bar{g}^{\sigma \alpha}\bar{\bar{g}}_{\alpha\gamma}\bar{g}^{\gamma q}-\frac{1}{2}\sum_{l\beta pq}\sum_k\frac{\partial f_1}{\partial x_k}\Big[-\bar{g}^{k\lambda}\bar{\bar{g}}_{\lambda\sigma}\bar{g}^{\sigma j}\Big(\frac{\partial \bar{\bar{g}}_{\beta j}}{\partial{x_l}}+\frac{\partial \bar{\bar{g}}_{lj}}{\partial{x_\beta}}-\frac{\partial \bar{\bar{g}}_{\beta l}}{\partial{x_j}}\Big)\nonumber\\
&+\bar{g}^{k\lambda}\bar{\bar{g}}_{\lambda\sigma}\bar{g}^{\sigma \alpha}\bar{\bar{g}}_{\alpha\gamma}\bar{g}^{\gamma j}\Big(\frac{\partial \bar{g}^{p j}}{\partial{x_l}}+\frac{\partial\bar{g}^{lj}}{\partial{x_\beta}}-\frac{\partial \bar{g}^{l\beta}}{\partial{x_j}}\Big)\Big]\Big(\frac{\partial^2f_2}{\partial x_p\partial x_q}-\sum_m\frac{\partial f_2}{\partial x_m}\bar{\Gamma}^m_{pq}\Big)\bar{g}^{\beta p}\bar{g}^{lq}-\frac{1}{2}\sum_{l\beta pq}\Big(\frac{\partial^2f_1}{\partial x_l\partial x_\beta}-\sum_k\frac{\partial f_1}{\partial x_k}\nonumber\\
&\bar{\Gamma}^k_{\beta l}\Big)\sum_m\frac{\partial f_2}{\partial x_m}\Big[-\bar{g}^{m\lambda}\bar{\bar{g}}_{\lambda\sigma}\bar{g}^{\sigma l}\Big(\frac{\partial \bar{\bar{g}}_{p l}}{\partial{x_q}}+\frac{\partial \bar{\bar{g}}_{ql}}{\partial{x_p}}-\frac{\partial \bar{\bar{g}}_{pq}}{\partial{x_l}}\Big)+\bar{g}^{m\lambda}\bar{\bar{g}}_{\lambda\sigma}\bar{g}^{\sigma \alpha}\bar{\bar{g}}_{\alpha\gamma}\bar{g}^{\gamma l}\Big(\frac{\partial \bar{g}^{p l}}{\partial{x_q}}+\frac{\partial\bar{g}^{ql}}{\partial{x_p}}-\frac{\partial \bar{g}^{pq}}{\partial{x_l}}\Big)\Big]\bar{g}^{\beta p}\bar{g}^{lq}\nonumber\\
&+\sum_{l\beta pq}\Big(\frac{\partial^2f_1}{\partial x_l\partial x_\beta}-\sum_k\frac{\partial f_1}{\partial x_k}\bar{\Gamma}^k_{\beta l}\Big)\Big(\frac{\partial^2f_2}{\partial x_p\partial x_q}-\sum_m\frac{\partial f_2}{\partial x_m}\bar{\Gamma}^m_{pq}\Big)\bar{g}^{\beta\lambda}\bar{\bar{g}}_{\lambda\sigma}\bar{g}^{\sigma p}\bar{g}^{l\lambda}\bar{\bar{g}}_{\lambda\sigma}\bar{g}^{\sigma q}+\frac{1}{2}\sum_{l\beta pq}\sum_k\frac{\partial f_1}{\partial x_k}\Big[\bar{g}^{kj}\Big(\frac{\partial \bar{\bar{g}}_{\beta j}}{\partial{x_l}}\nonumber\\
&+\frac{\partial \bar{\bar{g}}_{lj}}{\partial{x_\beta}}-\frac{\partial \bar{\bar{g}}_{\beta l}}{\partial{x_j}}\Big)-\bar{g}^{k\lambda}\bar{\bar{g}}_{\lambda\sigma}\bar{g}^{\sigma j}\Big(\frac{\partial \bar{g}^{\beta j}}{\partial{x_l}}+\frac{\partial\bar{g}^{lj}}{\partial{x_\beta}}-\frac{\partial \bar{g}^{\beta l}}{\partial{x_j}}\Big)\Big]\Big(\frac{\partial^2f_2}{\partial x_p\partial x_q}-\sum_m\frac{\partial f_2}{\partial x_m}\bar{\Gamma}^m_{pq}\Big)\Big(\bar{g}^{\beta\lambda}\bar{\bar{g}}_{\lambda\sigma}\bar{g}^{\sigma p}\bar{g}^{lq}+\bar{g}^{\beta p}\bar{g}^{l\lambda}\nonumber\\
&\bar{\bar{g}}_{\lambda\sigma}\bar{g}^{\sigma q}\Big)+\frac{1}{2}\Big(\frac{\partial^2f_1}{\partial x_l\partial x_\beta}-\sum_{l\beta pq}\sum_k\frac{\partial f_1}{\partial x_k}\bar{\Gamma}^k_{\beta l}\Big)\sum_m\frac{\partial f_2}{\partial x_m}\Big[\bar{g}^{ml}\Big(\frac{\partial \bar{\bar{g}}_{p l}}{\partial{x_q}}+\frac{\partial \bar{\bar{g}}_{q l}}{\partial{x_p}}-\frac{\partial \bar{\bar{g}}_{pq}}{\partial{x_l}}\Big)-\bar{g}^{m\lambda}\bar{\bar{g}}_{\lambda\sigma}\bar{g}^{\sigma l}\Big(\frac{\partial \bar{g}^{p l}}{\partial{x_q}}+\frac{\partial\bar{g}^{q l}}{\partial{x_p}}\nonumber\\
&-\frac{\partial \bar{g}^{pq}}{\partial{x_l}}\Big)\Big]\Big(\bar{g}^{\beta\lambda}\bar{\bar{g}}_{\lambda\sigma}\bar{g}^{\sigma p}\bar{g}^{lq}+\bar{g}^{\beta p}\bar{g}^{l\lambda}\bar{\bar{g}}_{\lambda\sigma}\bar{g}^{\sigma q}\Big)+\frac{1}{4}\sum_{l\beta pq}\sum_k\frac{\partial f_1}{\partial x_k}\Big[\bar{g}^{kj}\Big(\frac{\partial \bar{\bar{g}}_{\beta j}}{\partial{x_l}}+\frac{\partial \bar{\bar{g}}_{lj}}{\partial{x_\beta}}-\frac{\partial \bar{\bar{g}}_{\beta l}}{\partial{x_j}}\Big)-\bar{g}^{k\lambda}\bar{\bar{g}}_{\lambda\sigma}\bar{g}^{\sigma j}\Big(\frac{\partial \bar{g}^{\beta j}}{\partial{x_l}}\nonumber\\
&+\frac{\partial\bar{g}^{lj}}{\partial{x_\beta}}-\frac{\partial \bar{g}^{\beta l}}{\partial{x_j}}\Big)\Big]\sum_m\frac{\partial f_2}{\partial x_m}\Big[\bar{g}^{ml}\Big(\frac{\partial \bar{\bar{g}}_{p l}}{\partial{x_q}}+\frac{\partial \bar{\bar{g}}_{q l}}{\partial{x_p}}-\frac{\partial \bar{\bar{g}}_{pq}}{\partial{x_l}}\Big)-\bar{g}^{m\lambda}\bar{\bar{g}}_{\lambda\sigma}\bar{g}^{\sigma l}\Big(\frac{\partial \bar{g}^{p l}}{\partial{x_q}}+\frac{\partial\bar{g}^{q l}}{\partial{x_p}}-\frac{\partial \bar{g}^{pq}}{\partial{x_l}}\Big)\Big]\bar{g}^{\beta p}\bar{g}^{lq}.
\end{align}

$\mathbf{(4)}$The coefficients of $\varepsilon^0,\varepsilon,\varepsilon^2$ in $\Delta f_1\Delta f_2$:
\begin{align}\label{b4}
d_0&=\Big[\sum_{\alpha\beta}\bar{g}^{\alpha\beta}(x)\Big(\partial_\alpha\partial_\beta-\sum_l\bar{\Gamma}^l_{\alpha\beta}\partial_l\Big)(f_1)\Big]\Big[\sum_{pq}\bar{g}^{pq}(x)\Big(\partial_p\partial_q-\sum_r\bar{\Gamma}^r_{pq}\partial_r\Big)(f_2)\Big] ;\nonumber\\
d_1&=\Big[-\sum_{\alpha\beta}\bar{g}^{\alpha\lambda}\bar{\bar{g}}_{\lambda\sigma}\bar{g}^{\sigma \beta}\Big(\partial_\alpha\partial_\beta-\sum_l\bar{\Gamma}^l_{\alpha\beta}\partial_l\Big)(f_1)\Big]\Big[\sum_{pq}\bar{g}^{pq}(x)\Big(\partial_p\partial_q-\sum_r\bar{\Gamma}^r_{pq}\partial_r\Big)(f_2)\Big]+\Big[-\sum_{\alpha\beta}\bar{g}^{\alpha\beta}(x)\nonumber\\
&\Big(\partial_\alpha\partial_\beta-\sum_l\bar{\Gamma}^l_{\alpha\beta}\partial_l\Big)(f_1)\Big]\Big[\sum_{pq}\bar{g}^{p\lambda}\bar{\bar{g}}_{\lambda\sigma}\bar{g}^{\sigma q}\Big(\partial_p\partial_q-\sum_r\bar{\Gamma}^r_{pq}\partial_r\Big)(f_2)\Big] -\frac{1}{2}\Big[\sum_{\alpha\beta}\bar{g}^{\alpha\beta}(x)\Big(\partial_\alpha\partial_\beta-\sum_l\bar{\Gamma}^l_{\alpha\beta}\partial_l\Big)(f_1)\Big]\nonumber\\
&\sum_{pq}\bar{g}^{pq}(x)\Big[\bar{g}^{kr}\Big(\frac{\partial \bar{\bar{g}}_{pk}}{\partial{x_q}}+\frac{\partial \bar{\bar{g}}_{qk}}{\partial{x_p}}-\frac{\partial \bar{\bar{g}}_{pq}}{\partial{x_k}}\Big)-\bar{g}^{r\lambda}\bar{\bar{g}}_{\lambda\sigma}\bar{g}^{\sigma k}\Big(\frac{\partial \bar{g}^{pk}}{\partial{x_q}}+\frac{\partial\bar{g}^{qk}}{\partial{x_p}}-\frac{\partial \bar{g}^{pq}}{\partial{x_k}}\Big)\Big]\partial_r(f_2)-\frac{1}{2}\sum_{\alpha\beta}\bar{g}^{\alpha\beta}(x)\sum_l\Big[\bar{g}^{lk}\nonumber\\
&\Big(\frac{\partial \bar{\bar{g}}_{\alpha k}}{\partial{x_\beta}}+\frac{\partial \bar{\bar{g}}_{\beta k}}{\partial{x_\alpha}}-\frac{\partial \bar{\bar{g}}_{\alpha\beta}}{\partial{x_k}}\Big)-\bar{g}^{l\lambda}\bar{\bar{g}}_{\lambda\sigma}\bar{g}^{\sigma k}\Big(\frac{\partial \bar{g}^{lk}}{\partial{x_\beta}}+\frac{\partial\bar{g}^{\beta k}}{\partial{x_l}}-\frac{\partial \bar{g}^{l\beta}}{\partial{x_k}}\Big)\Big]\partial_l(f_1)\Big]\Big[\sum_{pq}\bar{g}^{pq}(x)\Big(\partial_p\partial_q-\sum_r\bar{\Gamma}^r_{pq}\partial_r\Big)(f_2)\Big] ;\nonumber\\
d_2&=\Big[\sum_{\alpha\beta}\bar{g}^{\alpha\lambda}\bar{\bar{g}}_{\lambda\sigma}\bar{g}^{\sigma \gamma}\bar{\bar{g}}_{\gamma m}\bar{g}^{m\beta}\Big(\partial_\alpha\partial_\beta-\sum_l\bar{\Gamma}^l_{\alpha\beta}\partial_l\Big)(f_1)\Big]\Big[\sum_{pq}\bar{g}^{pq}(x)\Big(\partial_p\partial_q-\sum_r\bar{\Gamma}^r_{pq}\partial_r\Big)(f_2)\Big]+\Big[\sum_{\alpha\beta}\bar{g}^{\alpha\beta}(x)\nonumber\\
&\Big(\partial_\alpha\partial_\beta-\sum_l\bar{\Gamma}^l_{\alpha\beta}\partial_l\Big)(f_1)\Big]\Big[\sum_{pq}\bar{g}^{p\lambda}\bar{\bar{g}}_{\lambda\sigma}\bar{g}^{\sigma \gamma}\bar{\bar{g}}_{\gamma m}\bar{g}^{mq}\Big(\partial_p\partial_q-\sum_r\bar{\Gamma}^r_{pq}\partial_r\Big)(f_2)\Big]-\frac{1}{2}\sum_{\alpha\beta}\bar{g}^{\alpha\beta}(x)\sum_l\Big[-\bar{g}^{l\lambda}\nonumber\\
&\bar{\bar{g}}_{\lambda\sigma}\bar{g}^{\sigma k}\Big(\frac{\partial \bar{\bar{g}}_{\alpha k}}{\partial{x_\beta}}+\frac{\partial \bar{\bar{g}}_{\beta k}}{\partial{x_\alpha}}-\frac{\partial \bar{\bar{g}}_{\alpha\beta}}{\partial{x_k}}\Big)+\bar{g}^{l\lambda}\bar{\bar{g}}_{\lambda\sigma}\bar{g}^{\sigma \gamma}\bar{\bar{g}}_{\gamma m}\bar{g}^{m\beta}\Big(\frac{\partial \bar{g}^{\alpha k}}{\partial{x_\beta}}+\frac{\partial\bar{g}^{\beta k}}{\partial{x_\alpha}}-\frac{\partial \bar{g}^{\alpha\beta}}{\partial{x_k}}\Big)\Big]\partial_l(f_1)\Big[\sum_{pq}\bar{g}^{pq}(x)\Big(\partial_p\partial_q\nonumber\\
&-\sum_r\bar{\Gamma}^r_{pq}\partial_r\Big)(f_2)\Big]-\frac{1}{2}\Big[\sum_{\alpha\beta}\bar{g}^{\alpha\beta}(x)\Big(\partial_\alpha\partial_\beta-\sum_l\bar{\Gamma}^l_{\alpha\beta}\partial_l\Big)(f_1)\Big]\sum_{pq}\bar{g}^{pq}(x)\sum_r\Big[-\bar{g}^{r\lambda}\bar{\bar{g}}_{\lambda\sigma}\bar{g}^{\sigma k}\Big(\frac{\partial \bar{\bar{g}}_{p k}}{\partial{x_q}}+\frac{\partial \bar{\bar{g}}_{q k}}{\partial{x_p}}\nonumber\\
&-\frac{\partial \bar{\bar{g}}_{pq}}{\partial{x_k}}\Big)+\bar{g}^{r\lambda}\bar{\bar{g}}_{\lambda\sigma}\bar{g}^{\sigma \gamma}\bar{\bar{g}}_{\gamma m}\bar{g}^{mq}\Big(\frac{\partial \bar{g}^{pk}}{\partial{x_q}}+\frac{\partial\bar{g}^{q k}}{\partial{x_p}}-\frac{\partial \bar{g}^{pq}}{\partial{x_k}}\Big)\Big]\partial_r(f_2)+\frac{1}{2}\sum_{pq}\bar{g}^{\alpha\lambda}\bar{\bar{g}}_{\lambda\sigma}\bar{g}^{\sigma \beta}\sum_l\Big[\bar{g}^{lk}\Big(\frac{\partial \bar{\bar{g}}_{\alpha k}}{\partial{x_\beta}}+\frac{\partial \bar{\bar{g}}_{\beta k}}{\partial{x_\alpha}}\nonumber\\
&-\frac{\partial \bar{\bar{g}}_{\alpha\beta}}{\partial{x_k}}\Big)-\bar{g}^{l\lambda}\bar{\bar{g}}_{\lambda\sigma}\bar{g}^{\sigma k}\Big(\frac{\partial \bar{g}^{lk}}{\partial{x_\beta}}+\frac{\partial\bar{g}^{\beta k}}{\partial{x_l}}-\frac{\partial \bar{g}^{l\beta}}{\partial{x_k}}\Big)\Big]\partial_l(f_1)\Big[\sum_{pq}\bar{g}^{pq}(x)\Big(\partial_p\partial_q-\sum_r\bar{\Gamma}^r_{pq}\partial_r\Big)(f_2)\Big]+\frac{1}{2}\Big[\sum_{\alpha\beta}\bar{g}^{\alpha\beta}(x)\nonumber\\
&\Big(\partial_\alpha\partial_\beta-\sum_l\bar{\Gamma}^l_{\alpha\beta}\partial_l\Big)(f_1)\Big]\sum_{pq}\bar{g}^{p\lambda}\bar{\bar{g}}_{\lambda\sigma}\bar{g}^{\sigma q}\sum_r\Big[\bar{g}^{kr}\Big(\frac{\partial \bar{\bar{g}}_{pk}}{\partial{x_q}}+\frac{\partial \bar{\bar{g}}_{qk}}{\partial{x_p}}-\frac{\partial \bar{\bar{g}}_{pq}}{\partial{x_k}}\Big)-\bar{g}^{r\lambda}\bar{\bar{g}}_{\lambda\sigma}\bar{g}^{\sigma k}\Big(\frac{\partial \bar{g}^{pk}}{\partial{x_q}}+\frac{\partial\bar{g}^{qk}}{\partial{x_p}}\nonumber\\
&-\frac{\partial \bar{g}^{pq}}{\partial{x_k}}\Big)\Big]\partial_r(f_2)+\Big[\sum_{\alpha\beta}\bar{g}^{\alpha\lambda}\bar{\bar{g}}_{\lambda\sigma}\bar{g}^{\sigma \beta}\Big(\partial_\alpha\partial_\beta-\sum_l\bar{\Gamma}^l_{\alpha\beta}\partial_l\Big)(f_1)\Big]\Big[\sum_{pq}\bar{g}^{p\lambda}\bar{\bar{g}}_{\lambda\sigma}\bar{g}^{\sigma q}\Big(\partial_p\partial_q-\sum_r\bar{\Gamma}^r_{pq}\partial_r\Big)(f_2)\Big]\nonumber\\
&+\frac{1}{2}\sum_{pq}\bar{g}^{\alpha\lambda}\bar{\bar{g}}_{\lambda\sigma}\bar{g}^{\sigma \beta}\Big(\partial_\alpha\partial_\beta-\sum_l\bar{\Gamma}^l_{\alpha\beta}\partial_l\Big)(f_1)\sum_{pq}\bar{g}^{pq}(x)\sum_r\Big[\bar{g}^{kr}\Big(\frac{\partial \bar{\bar{g}}_{pk}}{\partial{x_q}}+\frac{\partial \bar{\bar{g}}_{qk}}{\partial{x_p}}-\frac{\partial \bar{\bar{g}}_{pq}}{\partial{x_k}}\Big)-\bar{g}^{r\lambda}\bar{\bar{g}}_{\lambda\sigma}\bar{g}^{\sigma k}\nonumber\\
&\Big(\frac{\partial \bar{g}^{pk}}{\partial{x_q}}+\frac{\partial\bar{g}^{qk}}{\partial{x_p}}-\frac{\partial \bar{g}^{pq}}{\partial{x_k}}\Big)\Big]\partial_r(f_2)+\frac{1}{2}\sum_{\alpha\beta}\bar{g}^{\alpha\beta}(x)\sum_l\Big[\bar{g}^{lk}\Big(\frac{\partial \bar{\bar{g}}_{\alpha k}}{\partial{x_\beta}}+\frac{\partial \bar{\bar{g}}_{\beta k}}{\partial{x_\alpha}}-\frac{\partial \bar{\bar{g}}_{\alpha\beta}}{\partial{x_k}}\Big)-\bar{g}^{l\lambda}\bar{\bar{g}}_{\lambda\sigma}\bar{g}^{\sigma k}\Big(\frac{\partial \bar{g}^{lk}}{\partial{x_\beta}}\nonumber\\
&+\frac{\partial\bar{g}^{\beta k}}{\partial{x_l}}-\frac{\partial \bar{g}^{l\beta}}{\partial{x_k}}\Big)\Big] \Big[\sum_{pq}\bar{g}^{p\lambda}\bar{\bar{g}}_{\lambda\sigma}\bar{g}^{\sigma q}\Big(\partial_p\partial_q-\sum_r\bar{\Gamma}^r_{pq}\partial_r\Big)(f_2)\Big]+\frac{1}{4}\sum_{\alpha\beta}\bar{g}^{\alpha\beta}(x)\sum_l\Big[\bar{g}^{lk}\Big(\frac{\partial \bar{\bar{g}}_{\alpha k}}{\partial{x_\beta}}+\frac{\partial \bar{\bar{g}}_{\beta k}}{\partial{x_\alpha}}-\frac{\partial \bar{\bar{g}}_{\alpha\beta}}{\partial{x_k}}\Big)\nonumber\\
&-\bar{g}^{l\lambda}\bar{\bar{g}}_{\lambda\sigma}\bar{g}^{\sigma k}\Big(\frac{\partial \bar{g}^{lk}}{\partial{x_\beta}}+\frac{\partial\bar{g}^{\beta k}}{\partial{x_l}}-\frac{\partial \bar{g}^{l\beta}}{\partial{x_k}}\Big)\Big]\sum_{pq}\bar{g}^{pq}(x)\sum_r\Big[-\bar{g}^{r\lambda}\bar{\bar{g}}_{\lambda\sigma}\bar{g}^{\sigma k}\Big(\frac{\partial \bar{\bar{g}}_{p k}}{\partial{x_q}}+\frac{\partial \bar{\bar{g}}_{q k}}{\partial{x_p}}-\frac{\partial \bar{\bar{g}}_{pq}}{\partial{x_k}}\Big)+\bar{g}^{r\lambda}\bar{\bar{g}}_{\lambda\sigma}\bar{g}^{\sigma \gamma}\nonumber\\
&\bar{\bar{g}}_{\gamma m}\bar{g}^{mq}\Big(\frac{\partial \bar{g}^{pk}}{\partial{x_q}}+\frac{\partial\bar{g}^{q k}}{\partial{x_p}}-\frac{\partial \bar{g}^{pq}}{\partial{x_k}}\Big)\Big]\partial_r(f_2).
\end{align}

\section*{ Declarations}
\textbf{Ethics approval and consent to participate:} Not applicable.

\textbf{Consent for publication:} Not applicable.

\textbf{Availability of data and materials:} The authors confrm that the data supporting the fndings of this study are available within the article.

\textbf{Competing interests:} The authors declare no competing interests.

\textbf{Funding:} This research was funded by National Natural Science Foundation of China: No.11771070, 2023-BSBA-118 and N2405015..

\textbf{Author Contributions:} All authors contributed to the study conception and design. Material preparation,
data collection and analysis were performed by TW and YW. The frst draft of the manuscript was written
by TW and all authors commented on previous versions of the manuscript. All authors read and approved
the final manuscript.
\section*{Acknowledgements}
This work was supported by NSFC No.11771070, 2023-BSBA-118 and N2405015. The authors thank the referee for his (or her) careful reading and helpful comments.


\section*{References}
\begin{thebibliography}{00}
\bibitem{B2}Bochniak A. Towards modified bimetric theories within non-product spectral geometry. Journal of Physics A: Mathematical and Theoretical. 2022, 55(41). Paper No. 414006, 10 pp.
\bibitem{B3}Bochniak A, Sitarz A. Stability of Friedmann-Lema\^{\i}tre-Robertson-Walker solutions in doubled geometries. Physical review D. 2021, 103(4). Paper No. 044041,  13 pp.
\bibitem{B1}Brizuela D, Mart\'{\i}n-Garc\'{\i}a M J, Mena Marug\'{a}n G A. Second and higher-order perturbations of a spherical spacetime. American Physical Society. 2006, (4).
\bibitem{Ch2}Chandrasekhar S. The mathematical theory of black holes. The Clarendon Press, Oxford University Press. 1983, xxi+646 pp.
\bibitem{Ch1}Chen B Y, Yano K. On the theory of normal variations. Journal of Differential Geometry. 1978, 13: 1-10.
\bibitem{Co1}Connes A. Quantized calculus and applications. 11th International Congress of Mathematical Physics (Paris,1994). Internat Press, Cambridge, MA. 1995, 15-36.
\bibitem{D1}Dabrowski L, Sitarz A, Zalecki P. Spectral metric and Einstein functionals for Hodge-Dirac operator. Preprint 2023, arXiv: 2307.14877.
\bibitem{G1}Garc\'{\i}a-Bellido J, Wands D. Metric perturbations in two-field inflation. Physical review D. Particles and fields. 1995, 53(10): 5437-5445.
\bibitem{G2}Gundlach C. Critical phenomena in gravitational collapse. Physics Reports. 2003, 376(6): 339-405.
\bibitem{G3}Gundlach C, Calabrese G, Hinder I, Mart\'{\i}n-Garc\'{\i}a J M. Constraint damping in the $Z_4$ formulation and harmonic gauge. Classical Quantum Gravity. 2005, 22(17): 3767-3773.
\bibitem{H1}Herceg N, Juri T, Samsarov A, et al. Metric perturbations in noncommutative gravity. Journal of High Energy Physics. 2024, 2024(6): 130.
\bibitem{H2}Hirschfelder J O , Brown W B , Epstein S T. Recent Developments in Perturbation Theory. Advances in Quantum Chemistry. 1964, 1: 255-374.
\bibitem{K1}Khokhlov A M, Novikov I D. Gauge stability of 3+1 formulations of general relativity. Classical Quantum Gravity. 2002, 19(4): 827-846.
\bibitem{UW2}Ugalde W J. Differential forms and the noncommutative residue. Journal of Geometry and Physics. 2008, 58(12): 1739-1751.
\bibitem{UW1}Ugalde W J. A construction of critical GJMS operators using Wodzicki's residue. Communications in Mathematical Physics. 2006, 261(3): 771-788.
\bibitem{Mi1}Michel J, Radoux F, Josef S, et al. Second Order Symmetries of the Conformal Laplacian. Symmetry Integrability and Geometry-methods and Applications. 2013, 10: 016.
\bibitem{Mi2}Mukhanov V F, Feldman H A, Brandenberger R H. Theory of cosmological perturbations. Physics Reports. 1992, 215(5-6): 203-333.
\bibitem{O1}Okounkova M, Scheel M A, Teukolsky S A. Evolving metric perturbations in dynamical Chern-Simons gravity. Physical Review D. 2019, 99(4): 044019.1-044019.14.
\bibitem{S1}Seidel E. Nonlinear impact of perturbation theory on numerical relativity. Classical Quantum Gravity. 2004, 21(3): S339-S349.
\end{thebibliography}
\end{document}